\documentclass[10pt,reqno,oneside]{amsart}
\usepackage{geometry}
\usepackage{comment}
\usepackage{amssymb}
\usepackage{amsfonts}
\usepackage{amsmath}
\usepackage{amsthm}
\usepackage{shuffle}
\usepackage{mathtools}
\usepackage{enumerate}

\usepackage{graphicx,hyperref}

\usepackage{mleftright}
\mleftright
\usepackage{tikz}
\usetikzlibrary{intersections,calc,arrows}
\mathtoolsset{showonlyrefs=true}
\numberwithin{equation}{section}

\newcommand{\sh}{\mathbin{\shuffle}}

\newcommand{\fH}{\mathfrak{H}}
\DeclareMathOperator{\wt}{\mathrm{wt}}
\DeclareMathOperator{\dep}{\mathrm{dep}}
\newcommand{\cI}{\mathcal{I}}
\newcommand{\cA}{\mathcal{A}}
\newcommand{\cS}{\mathcal{S}}

\newcommand{\cZ}{\mathcal{Z}}
\newcommand{\cR}{\mathcal{R}}
\newcommand{\bbQ}{\mathbb{Q}}
\newcommand{\bbZ}{\mathbb{Z}}
\newcommand{\bbR}{\mathbb{R}}
\newcommand{\bbC}{\mathbb{C}}
\newcommand{\ba}{\mathbf{a}}
\newcommand{\bb}{\mathbf{b}}
\newcommand{\be}{\mathbf{e}}

\newcommand{\bk}{\mathbf{k}}
\newcommand{\bl}{\mathbf{l}}
\newcommand{\bm}{\mathbf{m}}
\newcommand{\bn}{\mathbf{n}}
\newcommand{\bp}{\mathbf{p}}
\newcommand{\ilp}{\boldsymbol{p}}
\newcommand{\hcA}{\widehat{\cA}}
\newcommand{\hcS}{\widehat{\cS}}

\newcommand{\emp}{\varnothing}
\newcommand{\shift}{\mathrm{shift}}
\newcommand{\jump}[1]{\ensuremath{[\![#1]\!]}}
\newcommand{\ncjump}[1]{\ensuremath{\langle\!\langle#1\rangle\!\rangle}}
\newcommand{\KZ}{\mathrm{KZ}}
\newcommand{\ep}{\varepsilon}
\newcommand{\Ad}{\mathrm{Ad}}
\DeclareMathOperator{\reg}{reg}
\newcommand{\cRS}{\mathcal{RS}}
\newcommand{\hcRS}{\widehat{\cRS}}

\newcommand{\per}{\mathrm{per}}
\newcommand{\hper}{\widehat{\per}}

\theoremstyle{theorem}
\newtheorem{Thm}{Theorem}[section]
\newtheorem{Prop}[Thm]{Proposition}
\newtheorem{Lem}[Thm]{Lemma}
\newtheorem{Cor}[Thm]{Corollary}
\newtheorem{Conj}[Thm]{Conjecture}
\theoremstyle{definition}
\newtheorem{Def}[Thm]{Definition}
\theoremstyle{remark}
\newtheorem{Rem}[Thm]{Remark}

\title{On a lifting of $t$-adic symmetric multiple zeta values}
\author{Minoru Hirose and Hanamichi Kawamura}
\makeatletter
\@namedef{subjclassname@2020}{%
  \textup{2020} Mathematics Subject Classification}
\makeatother
\subjclass[2020]{11M32.}
\keywords{multiple zeta values, $t$-adic symmetric multiple zeta values}
\address[Minoru Hirose]{Institute For Advanced Research, Nagoya University, Furo-cho, Chikusa-ku, Nagoya, 464-8602, Japan}
\email{minoru.hirose@math.nagoya-u.ac.jp}
\address[Hanamichi Kawamura]{Department of Mathematics, Faculty of Science Division I, Tokyo University of Science, 1-3 Kagurazaka, Shinjuku-ku, Tokyo, 162-8601, Japan}
\email{1121026@ed.tus.ac.jp}
\allowdisplaybreaks
\begin{document}
\maketitle
\begin{abstract}
The $t$-adic symmetric multiple zeta value is a generalization of the symmetric multiple zeta value from the perspective of the Kaneko--Zagier conjecture. In this paper, we introduce a further generalization with a new parameter $s$, which we call the $(s,t)$-adic symmetric multiple zeta value. Then, the $(s,t)$-adic version of the $t$-adic double shuffle relations, duality and cyclic sum formula are established. A finite counterpart of the $(s,t)$-adic symmetric multiple zeta value is also discussed.
\end{abstract}
\section{Introduction}
An element of $\cI\coloneqq\bigsqcup_{r\ge 0}\bbZ_{\ge 1}^{r}$ is called an \emph{index}.
We say that an index $\bk=(k_{1},\ldots,k_{r})$ has the \emph{depth} $\dep(\bk)\coloneqq r$ and the \emph{weight} $\wt(\bk)=k_{1}+\cdots+k_{r}$.
The unique index with depth $0$ is the \emph{empty index} $\emp$.
For tuples of integers $\bk=(k_{1},\ldots,k_{r})$ and $\bl=(l_1,\ldots,l_r)$, we put
\begin{gather}
\bk_{[i]}\coloneqq(k_{1},\ldots,k_{i}),\quad\bk^{[i]}\coloneqq(k_{i+1},\ldots,k_{r})\qquad(0\le i\le r),\qquad\overleftarrow{\bk}=(k_{r},\ldots,k_{1}),\\
\bk\oplus\bl\coloneqq(k_1+l_1,\ldots,k_r+l_r),\quad b\left({\bk\atop\bl}\right)\coloneqq\prod_{i=1}^{r}\binom{k_{i}+l_{i}-1}{l_{i}}.
\end{gather}
Note that $\overleftarrow{\emp}=\emp$ and $\bk_{[0]}=\bk^{[r]}=\emp$. Furthermore, we denote by $(\bk,\bl)$ the concatenation of indices $\bk$ and $\bl$.
If an index $\bk$ has the last component greater than $1$ or $\bk=\emp$, we call $\bk$ an \emph{admissible} index.
The \emph{multiple zeta value} (MZV) for an admissible index $\bk=(k_1,\ldots,k_r)$ is defined by
\[\zeta(\bk)\coloneqq\sum_{0 < n_{1} < \cdots < n_{r}}\frac{1}{n_{1}^{k_{1}}\cdots n_{r}^{k_{r}}},\qquad\zeta(\emp)\coloneqq1.\]
Several variants of these values has been studied. One such variant, proposed by Kaneko--Zagier \cite{kz21}, is the \emph{symmetric multiple zeta value} (SMZV)
\[\zeta_{\cS}(\bk)\coloneqq\sum_{i=0}^{\dep(\bk)}(-1)^{\wt(\bk^{[i]})}\zeta^{\bullet}(\bk_{[i]};T)\zeta^{\bullet}(\overleftarrow{\bk^{[i]}};T)\mod\zeta(2),\]
where $\zeta^{\bullet}(\bk;T)$ is the \emph{regularized polynomial} introduced by Ihara--Kaneko--Zagier \cite{ikz06} (for details, see Section \ref{sec:preliminaries}).
Jarossay \cite{jarossay17}, Hirose--Murahara--Ono \cite{hmo21} and Ono--Seki--Yamamoto \cite{osy21} defined a generalization of SMZVs, called the \emph{$t$-adic symmetric multiple zeta value}
\[\zeta_{\hcS}(\bk)\coloneqq\sum_{i=0}^{\dep(\bk)}(-1)^{\wt(\bk^{[i]})}\zeta^{\bullet}(\bk_{[i]};T)\zeta^{-t,\bullet}_{\shift}(\overleftarrow{\bk^{[i]}};T)~\mathrm{mod}~\zeta(2).\]
The purpose of this paper is to give a further generalization of $t$-adic SMZVs by considering a new parameter $s$ and to lift several properties known for $t$-adic SMZVs.
We call our new object \emph{$(s,t)$-adic symmetric multiple zeta values}. Throughout this paper, $s$ and $t$ denote indeterminates.
\begin{Def}[$(s,t)$-adic symmetric multiple zeta value]\label{stadic_smzv}
For an index $\bk$, we define
\[\zeta^{s,t}_{\hcS}(\bk)\coloneqq\sum_{i=0}^{\dep(\bk)}(-1)^{\wt(\bk^{[i]})}\zeta^{s,\bullet}_{\shift}(\bk_{[i]};T)\zeta^{-t,\bullet}_{\shift}(\overleftarrow{\bk^{[i]}};T)\mod\zeta(2).\]
\end{Def}
In the next section, we prove the well-definedness (independence of $\bullet$ and $T$) of this definition.
Our main results are summarized as follows (see later sections for the precise definitions of some symbols):
\begin{Thm}[$(s,t)$-adic harmonic relation]\label{stadic_harmonic_relation}
For indices $\bk$ and $\bl$, we have
\[\zeta^{s,t}_{\hcS}(\bk\ast\bl)=\zeta^{s,t}_{\hcS}(\bk)\zeta^{s,t}_{\hcS}(\bl).\]
\end{Thm}
\begin{Thm}[$(s,t)$-adic shuffle relation]\label{stadic_shuffle_relation}
For indices $\bk$ and $\bl$, we have
\[\zeta^{s,t}_{\hcS}(\bk\sh_{s}\bl)=(-1)^{\wt(\bl)}\sum_{\bn\in\bbZ^{\dep(\bl)}_{\ge 0}}b\left({\bl\atop\bn}\right)\zeta^{s,t}_{\hcS}(\bk,\overleftarrow{\bl\oplus\bn})t^{\wt(\bn)}.\]
\end{Thm}
\begin{Thm}[$(s,t)$-adic duality]\label{stadic_duality}
For a non-empty index $\bk$, we have
\[\sum_{m,n=0}^{\infty}\zeta^{s,t,\star}_{\hcS}(\{1\}^{m},\bk,\{1\}^{n})s^{m}t^{n}=-\sum_{m,n=0}^{\infty}\zeta^{s,t,\star}_{\hcS}(\{1\}^{m},\bk^{\vee},\{1\}^{n})s^{m}t^{n}.\]
\end{Thm}
\begin{Thm}[$(s,t)$-adic cyclic sum formula]\label{stadic_csf}
For an index $\bk=(k_{1},\ldots,k_{r})$ with $\wt(\bk)>r$, we have
\begin{align}
&\sum_{i=1}^{r}\sum_{j=0}^{k_{i}-1}\zeta^{s,t,\star}_{\hcS}(j+1,\bk^{[i]},\bk_{[i-1]},k_{i}-j)\\
&\qquad=\zeta^{s,t,\star}_{\hcS}(k+1)+\sum_{i=1}^{r}\sum_{j=0}^{\infty}(\zeta^{s,t,\star}_{\hcS}(j+1,\bk^{[i]},\bk_{[i]})s^{j}+\zeta^{s,t,\star}_{\hcS}(\bk^{[i]},\bk_{[i]},j+1)t^{j}).
\end{align}
\end{Thm}
More precisely, we prove them by generalizing to relations among values $\zeta^{s,t,\bullet}_{\hcS}(\bk;T_{1},T_{2})$, which includes the $(s,t)$-adic version of refined symmetric multiple zeta values (RSMZVs) (see \cite{hirose20}).

This paper is organized as follows. In Section \ref{sec:preliminaries}, we recall the algebraic formulation of MZVs and prove basic properties of $(s,t)$-adic SMZVs. Sections \ref{sec:harmonic_relation}--\ref{sec:cyclic_sum_formula} are used to prove Theorems \ref{stadic_harmonic_relation},  \ref{stadic_shuffle_relation}, \ref{stadic_duality} and \ref{stadic_csf} respectively. In the last section, we discuss the finite counterpart of $(s,t)$-adic SMZVs from the viewpoint of the Kaneko--Zagier conjecture.
\section*{Acknowledgements}
The authors would like to thank Prof.~Masataka Ono, Prof.~Shin-ichiro Seki and Takumi Maesaka for their helpful suggestions.
This work was supported by JSPS KAKENHI Grant Numbers JP18K13392 and JP22K03244.
This research was also supported by JST Global Science Campus ROOT Program.
\section{Preliminaries}\label{sec:preliminaries}
\subsection{Hoffman's algebraic formulation}\label{subsec:algebraic_formulation}
Let $\fH$ be the non-commutative algebra $\bbQ\langle e_{0},e_{1}\rangle$ and
\[\fH^{1}\coloneqq\bbQ+e_{1}\fH,\qquad\fH^{0}\coloneqq\bbQ+e_{1}\fH e_{0}\]
its subalgebras.
For an index $\bk=(k_{1},\ldots,k_{r})$, we set $e_{\bk}\coloneqq e_{1}e_{0}^{k_{1}-1}\cdots e_{1}e_{0}^{k_{r}-1}$ (the empty product is regarded as $1$).
For a set $S$, we write $\mathrm{span}_{\bbQ}S$ for the $\bbQ$-vector space of formal $\bbQ$-linear sums of elements of $S$.
If we write $\cR\coloneqq\mathrm{span}_{\bbQ}\cI$,
we have the $\bbQ$-linear bijection $\cR\to\fH^{1};\bk\to (-1)^{\dep(\bk)}e_{\bk}$.
Under this bijection, the subset $\cR'\coloneqq\mathrm{span}_{\bbQ}\cI'$, where $\cI'$ is the set of all admissible indices, corresponds to $\fH^{0}$.

We define two $\bbQ$-bilinear product structures on them: the first one is the \emph{harmonic product} $\ast$ on $\fH^{1}$, defined recursively as
\begin{gather}
1\ast w=w\ast 1=w,\\
we_{(k)}\ast w'e_{(l)}=(we_{(k)}\ast w')e_{(l)}+(w\ast w'e_{(l)})e_{(k)}-(w\ast w')e_{(k+l)}
\end{gather}
for $w,w'\in\fH^{1}$ and $k,l\in\bbZ_{\ge 1}$. The other one is the \emph{shuffle product} $\sh$ on $\fH$, also defined by inductive rules
\begin{gather}
1\sh w=w\sh 1=w,\\
we_{i}\sh w'e_{j}=(we_{i}\sh w')e_{j}+(w\sh w'e_{j})e_{i}
\end{gather}
where $w,w'\in\fH$ and $i,j\in\{0,1\}$. Through the correspondence $e_{\bk}\mapsto(-1)^{\dep(\bk)}\bk$, we regard these as products on $\cR$.
Then the harmonic relation (resp. shuffle relation) for multiple zeta values states that
\[\zeta(\bk\ast\bl)=\zeta(\bk)\zeta(\bl)\qquad (\text{resp.}~\zeta(\bk\sh\bl)=\zeta(\bk)\zeta(\bl))\]
for any admissible indices $\bk$ and $\bl$.
Here and in what follows, we $\bbQ$-linearly extend a map whose argument is an index.
It is known that $\fH^{i}_{\bullet}\coloneqq(\fH^{i},\bullet)$ becomes an associative and commutative $\bbQ$-algebra for the product $\bullet\in\{\ast,\sh\}$ and $i\in\{0,1\}$.
We define the $\bbQ$-linear map $Z\colon\fH^{0}\to\bbR$ by $Z(e_{\bk})=(-1)^{\dep(\bk)}\zeta(\bk)$ ($\bk\in\cI'$). Then the harmonic and shuffle relations are rephrased as that $Z\colon\fH^{0}_{\bullet}\to\bbR$ is a $\bbQ$-algebra homomorphism.

We now review the definition of regularized polynomials stated in \cite{ikz06}. In what follows, we use the symbol $\bullet$ as either $\ast$ or $\sh$.
Since the isomorphism $\fH^{1}_{\bullet}\simeq\fH^{0}_{\bullet}[e_{1}]$ proved by Hoffman \cite[Theorem 3.1]{hoffman97} and Reutenauer \cite[Theorem 6.1]{reutenauer93} exists, there uniquely exists a $\bbQ$-algebra homomorphism $Z^{\bullet}_{T}\colon\fH^{1}_{\bullet}\to\bbR[T]$ such that $Z^{\bullet}_T|_{\fH^{0}}=Z$ and $Z^{\bullet}_T(e_{1})=-T$.
For an index $\bk$, we put $\zeta^{\bullet}(\bk;T)\coloneqq (-1)^{\dep(\bk)}Z^{\bullet}_{T}(e_{\bk})$, which is called the \emph{$\bullet$-regularized polynomial}.
We abbreviate these notations as $Z^{\bullet}(w)$ and $\zeta^{\bullet}(\bk)$ when $T=0$.
Moreover, using the isomorphism $\fH_{\sh}\simeq\fH^{1}_{\sh}[e_{0}]$, we extend the map $Z^{\sh}_{T}$ as a $\bbQ$-algebra homomorphism $Z^{\sh}_{T}\colon(\fH,\sh)\to\bbR[T]$ satisfying $Z^{\sh}_{T}(e_{0})=0$.
We remark that
\begin{equation}\label{eq:regularization_whole_fH}
Z^{\sh}_{T}(e_{0}^{n}e_{\bk})=(-1)^{n+\dep(\bk)}\sum_{\substack{\bn\in\bbZ_{\ge 0}^{\dep(\bk)}\\\wt(\bn)=n}}b\left({\bk\atop\bn}\right)\zeta^{\sh}(\bk\oplus\bn;T)
\end{equation}
for a non-negative integer $n$ and an index $\bk$.

Recall the regularization theorem of Ihara--Kaneko--Zagier \cite{ikz06}.
Define an $\bbR$-linear map $\rho\colon\bbR[T]\to\bbR[T]$ by
\[\exp(TX)\Gamma_{0}(-X)=\sum_{n=0}^{\infty}\frac{\rho(T^{n})}{n!}X^{n},\]
where $\Gamma_{0}$ is the formal power series
\[\Gamma_{0}(X)\coloneqq\exp\left(\sum_{k=2}^{\infty}\frac{\zeta(k)}{k}X^{k}\right)\in\bbR\jump{X}.\]
\begin{Thm}[Ihara--Kaneko--Zagier {\cite[Theorem 1]{ikz06}}]\label{ikz_reg_thm}
For an index $\bk$, we have
\[\zeta^{\sh}(\bk;T)=\rho(\zeta^{\ast}(\bk;T)).\]
\end{Thm}
\subsection{Definition of $(s,t)$-adic SMZVs}
\begin{Def}[Shifted multiple zeta values]
For an index $\bk$, we define
\[\zeta^{t,\bullet}_{\shift}(\bk;T)\coloneqq\sum_{\bn\in\bbZ^{\dep(\bk)}_{\ge 0}}b\left(\bk\atop\bn\right)\zeta^{\bullet}(\bk\oplus\bn;T)(-t)^{\wt(\bn)}\in\bbR\jump{t}.\]
We omit the symbol $T$ when $T=0$.
\end{Def}
\begin{Rem}
When $\bk$ is admissible, this value coincides with the multiple zeta value of Hurwitz type
\[\zeta^{(t)}(\bk)\coloneqq\sum_{0<n_{1}<\cdots<n_{r}}\frac{1}{(n_{1}+t)^{k_{1}}\cdots (n_{r}+t)^{k_{r}}}\in\bbR\jump{t}.\]
More generally, even if $\bk$ is a non-admissible index, it matches the regularization of $\zeta^{(t)}$ defined by Kaneko--Xu--Yamamoto \cite{kxy20}.
Using their notation, we can describe this coincidence as $Z_{\ast}^{(t)}(\bk;T-\gamma-\psi(1+t))=\zeta^{t,\ast}_{\shift}(\bk;T)$ and $P^{(t)}(\bk;T)=\zeta^{t,\sh}_{\shift}(\bk;T)$.
\end{Rem}
In order to check the well-definedness of Definition \ref{stadic_smzv}, we consider the value
\[\zeta^{s,t,\bullet}_{\hcS}(\bk;T_{1},T_{2})\coloneqq\sum_{i=0}^{\dep(\bk)}(-1)^{\wt(\bk^{[i]})}\zeta^{s,\bullet}_{\shift}(\bk_{[i]};T_{1})\zeta^{-t,\bullet}_{\shift}(\overleftarrow{\bk^{[i]}};T_{2})\in\cZ[T_{1},T_{2}]\jump{s,t}\]
for an index $\bk$, where $\cZ$ denotes the $\bbQ$-algebra generated by all MZVs.
The $\bbQ$-linear operator $Z^{s,t,\bullet}_{\hcS,T_{1},T_{2}}\colon\fH^{1}\to\cZ[T_{1},T_{2}]\jump{s,t}$ is defined by $e_{\bk}\mapsto(-1)^{\dep(\bk)} \zeta^{s,t,\bullet}_{\hcS}(\bk;T_{1},T_{2})$.
Let $\{X_{0},X_{1}\}^{\times}$ be the set of all words (including the empty word $1$) consisting of $X_{0}$ and $X_{1}$,
and $\dep$ (resp. $\wt$) the map giving the number of $X_{1}$'s (resp. letters) in each word.
For an element $W=X_{a_{1}}\cdots X_{a_{n}}$ of $\{X_{0},X_{1}\}^{\times}$,
put $\overleftarrow{W}\coloneqq X_{a_{n}}\cdots X_{a_{1}}$ and extend it to $R\ncjump{X_0,X_1}$ for any $\bbQ$-algebra $R$ by
\[\overleftarrow{\sum_{W\in\{X_{0},X_{1}\}^{\times}}a_{W}W}\coloneqq\sum_{W\in\{X_{0},X_{1}\}^{\times}}a_{W}\overleftarrow{W}\qquad(a_{W}\in R).\]
We define a generating function of regularized polynomials by
\[\Phi^{\bullet}(X_{0},X_{1};T)\coloneqq\sum_{\substack{a_{1},\ldots,a_{n}\in\{0,1\}\\ n\ge 0}}Z^{\bullet}_{T}(e_{a_{1}}\cdots e_{a_{n}})X_{a_{n}}\cdots X_{a_{1}},\]
where $Z^{\ast}_T(w)$ ($w\in\fH$) is defined similarly to \eqref{eq:regularization_whole_fH} (replace $\sh$ by $\ast$).
Then the shifted multiple zeta value is written as
\begin{equation}\label{eq:shifted_mzv_coefficient}
\zeta^{t,\bullet}_{\shift}(\bk;T)=(-1)^{\dep(\bk)}\sum_{n=0}^{\infty}Z^{\bullet}_T(e_{0}^{n}e_{\bk})t^{n}\qquad(\bk\in\cI)
\end{equation}
and thus we have
\[\Phi^{\bullet}(X_{0},X_{1};T)=\sum_{n=0}^{\infty}\sum_{\bk\in\cI}Z^{\bullet}_{T}(e_{0}^{n}e_{\bk})\overleftarrow{X_{0}^{n}X_{\bk}}=\sum_{\bk\in\cI}(-1)^{\dep(\bk)}\overleftarrow{\zeta^{X_{0},\bullet}_{\shift}(\bk;T)X_{\bk}},\]
where $X_{\bk}\coloneqq X_{1}X_{0}^{k_{1}-1}\cdots X_{1}X_{0}^{k_{r}-1}$ for $\bk=(k_{1},\ldots,k_{r})\in\cI$.
\begin{Lem}\label{sigma_harmonic_homomorphism}
Let $\sigma^{t}$ be the ring homomorphism $\fH\to\fH\jump{t}$ defined by $e_{i}\mapsto e_{i}(1+e_{0}t)^{-1}$ for $i\in\{0,1\}$ and extend the harmonic product on $\fH^{1}\jump{t}$ by
\[\left(\sum_{m\ge 0}w_{m}t^{m}\right)\ast\left(\sum_{n\ge 0}w'_{n}t^{n}\right)=\sum_{m,n\ge 0}(w_{m}\ast w'_{n})t^{m+n}.\]
Then $\sigma^{t}$ induces the $\bbQ$-algebra homomorphism $\fH^{1}_{\ast}\to\fH^{1}_{\ast}\jump{t}$.
\end{Lem}
\begin{proof}
We prove $\sigma^{t}(e_{\bk}\ast e_{\bl})=\sigma^{t}(e_{\bk})\ast\sigma^{t}(e_{\bl})$ for indices $\bk$ and $\bl$ by induction on $N=\dep(\bk)+\dep(\bl)$.
We assume the case $N<m$ and prove the case $N=m$.
When $\dep(\bk)=0$ or $\dep(\bl)=0$, this identity is clear from the definition of $\ast$.
For indices $\bk,\bl$ with the sum of depths $m-2$ and positive integers $k$ and $l$, the inductive hypothesis establishes
\begin{align}
&\sigma^{t}(e_{\bk}e_{(k)}\ast e_{\bl}e_{(l)})\\&\qquad=(\sigma^{t}(e_{\bk}e_{(k)})\ast\sigma^{t}(e_{\bl}))\sigma^{t}(e_{(l)})+(\sigma^{t}(e_{\bk})\ast\sigma^{t}(e_{\bl}e_{(l)}))\sigma^{t}(e_{(k)})-(\sigma^{t}(e_{\bk})\ast\sigma^{t}(e_{\bl}))\sigma^{t}(e_{(k+l)}).
\end{align}
On the other hand, we have
\begin{align}
\sigma^{t}(e_{\bk}e_{(k)})\ast\sigma^{t}(e_{\bl}e_{(l)})
&=\sum_{m,n=0}^{\infty} \binom{k+m-1}{m}\binom{l+n-1}{n}(\sigma^t(e_{\bk})e_{(k+m)}\ast\sigma^{t}(e_{\bl})e_{(l+n)})(-t)^{m+n}\\
&=\begin{multlined}[t]
\sum_{m,n=0}^{\infty} \binom{k+m-1}{m}\binom{l+n-1}{n}\biggl((\sigma^{t}(e_{\bk})e_{(k+m)}\ast\sigma^{t}(e_{\bl}))e_{(l+n)}\\
+(\sigma^{t}(e_{\bk})\ast\sigma^{t}(e_{\bl})e_{(l+n)})e_{(k+m)}-(\sigma^{t}(e_{\bk})\ast\sigma^{t}(e_{\bl}))e_{(k+m+l+n)}\biggr)(-t)^{m+n}
\end{multlined}\\
&=\begin{multlined}[t]
(\sigma^{t}(e_{\bk}e_{(k)})\ast\sigma^{t}(e_{\bl}))\sigma^{t}(e_{(l)})+(\sigma^{t}(e_{\bk})\ast\sigma^{t}(e_{\bl}e_{(l)}))\sigma^{t}(e_{(k)})\\
-(\sigma^{t}(e_{\bk})\ast\sigma^{t}(e_{\bl}))\cdot\sum_{m,n=0}^{\infty}\binom{k+m-1}{m}\binom{l+n-1}{n}e_{(k+m+l+n)}(-t)^{m+n}.
\end{multlined}
\end{align}
The last series can be computed as
\[\sum_{m,n=0}^{\infty}\binom{k+m-1}{m}\binom{l+n-1}{n}e_{(k+m+l+n)}(-t)^{m+n}=e_{(k+l)}(1+e_{0}t)^{-k}(1+e_{0}t)^{-l}=\sigma^{t}(e_{(k+l)})\]
and thus we obtain the desired lemma.
\end{proof}
Combining $\zeta^{\ast}(\bk\ast\bl;T)=\zeta^{\ast}(\bk;T)\zeta^{\ast}(\bl;T)$ ($\bk,\bl\in\cI$) and Lemma \ref{sigma_harmonic_homomorphism}, we get the harmonic relation for $\zeta^{t,\ast}_{\shift}(\bk;T)$.
\begin{Cor}\label{shifted_harmonic_relation}
For an index $\bk$ and $\bl$, we have $\zeta^{t,\ast}_{\shift}(\bk\ast\bl;T)=\zeta^{t,\ast}_{\shift}(\bk;T)\zeta^{t,\ast}_{\shift}(\bl;T)$.
\end{Cor}
From this harmonic relation and an explicit formula for the antipode of $\fH^{1}$ in Proposition \ref{harmonic_hopf}, we obtain the following:
\begin{Cor}\label{antipode_relation}
For an index $\bk$, we have
\[\sum_{i=0}^{\dep(\bk)}(-1)^{i}\zeta^{t,\ast}_{\shift}(\overleftarrow{\bk_{[i]}};T)\zeta^{t,\star,\ast}_{\shift}(\bk^{[i]};T)=\delta_{0,\dep(\bk)},\]
where $\delta_{i,j}$ is Kronecker's delta and $\zeta^{t,\star,\ast}_{\shift}$ is defined as in Section \ref{sec:cyclic_sum_formula}.
\end{Cor}
For $w\in\fH^{1}$, we denote by $\reg_{\ast}(w)$ the constant term of $w$ with respect to the isomorphism $\fH^{1}_{\ast}\simeq\fH^{0}_{\ast}[e_{1}]$.
Then the equality $Z^{\ast}=Z\circ\reg_{\ast}$ holds by definition.
\begin{Lem}[{Ihara--Kaneko--Zagier \cite[Corollary 5]{ikz06}}]\label{ikz_cor5}
For a non-negative integer $n$ and $w\in\fH^{0}$, we have
\[we_{1}^{n}=\sum_{i=0}^{n}\frac{\reg_{\bullet}(we_{1}^{n-i})}{i!}\bullet\underbrace{e_{1}\bullet\cdots\bullet e_{1}}_{i}.\]
\end{Lem}
\begin{Prop}\label{T_part}
We have
\[\Phi^{\bullet}(X_{0},X_{1};T)=\exp(-TX_{1})\Phi^{\bullet}(X_{0},X_{1};0).\]
\end{Prop}
\begin{proof}
For a non-negative integer $n$, $\{X\}^n$ denotes $n$ times repetition of $X$. Lemma \ref{ikz_cor5} shows
\begin{align}
\overleftarrow{\Phi^{\bullet}(X_{0},X_{1};T)}
&=\sum_{\bk\in\cI'}\sum_{n=0}^{\infty}(-1)^{\dep(\bk)+n}\zeta^{X_{0},\bullet}_{\shift}(\bk,\{1\}^{n};T)X_{\bk}X_{1}^{n}\\
&=\sum_{\bk\in\cI'}\sum_{n=0}^{\infty}\sum_{i=0}^{n}(-1)^{\dep(\bk)+n-i}\zeta^{X_{0},\bullet}_{\shift}(\bk,\{1\}^{n-i})X_{\bk}X_{1}^{n}\frac{(-T)^{i}}{i!}\\
&=\sum_{\bk\in\cI'}\sum_{i=0}^{\infty}\sum_{n=0}^{\infty}(-1)^{\dep(\bk)+n}\zeta^{X_{0},\bullet}_{\shift}(\bk,\{1\}^{n})X_{\bk}X_{1}^{n+i}\frac{(-T)^{i}}{i!}\\
&=\overleftarrow{\Phi^{\bullet}(X_{0},X_{1};0)}\exp(-TX_{1}),
\end{align}
and hence we get the conclusion.
\end{proof}
\begin{Prop}\label{reg_thm_associator}
We have
\[\Phi^{\sh}(X_{0},X_{1};T)= \Gamma_{0}(X_{1}) \Phi^{\ast}(X_{0},X_{1};T).\]
\end{Prop}
\begin{proof}
By Theorem \ref{ikz_reg_thm} and Proposition \ref{T_part}, we have
\[
\Phi^{\sh}(X_{0},X_{1};T) = \rho(\Phi^{\ast}(X_{0},X_{1};T)) = \rho(\exp(-TX_{1}))\Phi^{\ast}(X_{0},X_{1};0)=\Gamma_{0}(X_1)\Phi^{\ast}(X_{0},X_{1};T). \qedhere
\]
\end{proof}
Let $\ep$ be the anti-automorphism of $\bbR[T]\ncjump{X_{0},X_{1}}$ determined by $\ep(T)=T$ and $\ep(X_{i})=-X_{i}$ ($i\in\{0,1\}$),
and put
\[\Phi^{\bullet}_{\Ad}(X_{0},X_{1};T_{1},T_{2})\coloneqq\ep(\Phi^{\bullet}(X_{0},X_{1};T_{1}))X_{1}\Phi^{\bullet}(X_{0},X_{1};T_{2}).\]
For a commutative $\bbQ$-algebra $A$, we define the pairing $\langle\Phi,w\rangle\in A$ of the power series $\Phi\in A\ncjump{X_{0},X_{1}}$ and $w\in \fH$ by
\[\Phi=\sum_{\substack{a_{1},\ldots,a_{n}\in\{0,1\}\\ n\ge 0}}\langle\Phi,e_{a_1}\cdots e_{a_n}\rangle X_{a_1}\cdots X_{a_n}\]
and linearity.
Moreover, this is extended to $\fH\jump{s,t}$ as
\[\left\langle\Phi,~\sum_{m,n=0}^{\infty}w_{m,n}s^{m}t^{n}\right\rangle\coloneqq\sum_{m,n=0}^{\infty}\langle\Phi,~w_{m,n}\rangle s^{m}t^{n}\qquad(w_{m,n}\in\fH).\]
\begin{Prop}\label{stadic_smzv_associator}
For a non-empty index $\bk$, we define $\omega(\bk)\in\fH$ by $e_{1}\omega(\bk)=e_{\bk}$. Then we have
\[\zeta^{s,t,\bullet}_{\hcS}(\bk;T_{1},T_{2})=(-1)^{\wt(\bk)+\dep(\bk)}\left\langle\Phi^{\bullet}_{\Ad}(X_{0},X_{1};T_{1},T_{2}),~\frac{1}{1+e_{0}s}e_{1}\omega(\bk)e_{1}\frac{1}{1+e_{0}t}\right\rangle.\]
\end{Prop}
\begin{proof}
Write $\bk=(k_1,\ldots,k_r)$.
By definition, we can compute as
\begin{align}
&\left\langle\Phi^{\bullet}_{\Ad}(X_{0},X_{1};T_{1},T_{2}),~\frac{1}{1+e_{0}s}e_{1}\omega(\bk)e_{1}\frac{1}{1+e_{0}t}\right\rangle\\
&=\sum_{m,n=0}^{\infty}\langle\Phi^{\bullet}_{\Ad}(X_{0},X_{1};T_{1},T_{2}),~e_{0}^{m}e_{1}e_{0}^{k_{1}-1}\cdots e_{1}e_{0}^{k_{r}-1}e_{1}e_{0}^{n}\rangle (-s)^{m}(-t)^{n}\\
&=\begin{multlined}[t]
\sum_{m,n=0}^{\infty}\sum_{i=0}^{r}\langle\ep(\Phi^{\bullet}(X_{0},X_{1};T_{1})),~e_{0}^{m}e_{1}e_{0}^{k_{1}-1}\cdots e_{1}e_{0}^{k_{i}-1}\rangle\\
\cdot\langle\Phi^{\bullet}(X_{0},X_{1};T_{2}),~e_{0}^{k_{i+1}-1}e_{1}\cdots e_{0}^{k_{r}-1}e_{1}e_{0}^{n}\rangle(-s)^{m}(-t)^{n}
\end{multlined}\\
&=\begin{multlined}[t]
\sum_{m,n=0}^{\infty}\sum_{i=0}^{r}(-1)^{m+\wt(\bk_{[i]})}Z^{\bullet}_{T_{1}}(e_{0}^{m}e_{1}e_{0}^{k_{1}-1}\cdots e_{1}e_{0}^{k_{i}-1})\\
\cdot Z^{\bullet}_{T_{2}}(e_{0}^{n}e_{1}e_{0}^{k_{r}-1}\cdots e_{1}e_{0}^{k_{i+1}-1})(-s)^{m}(-t)^{n}
\end{multlined}\\
&=(-1)^{\wt(\bk)+r}\zeta^{s,t,\bullet}_{\hcS}(\bk;T_{1},T_{2}).
\end{align}
Here, we used \eqref{eq:shifted_mzv_coefficient} in the last equality.
\end{proof}
The following proposition guarantees the well-definedness of Definition \ref{stadic_smzv}.
\begin{Prop}\label{stadic_independence}
For an index $\bk$, we have $\zeta^{s,t,\bullet}_{\hcS}(\bk;T_{1},T_{2})=\zeta^{s,t,\bullet}_{\hcS}(\bk;0,T_{2}-T_{1})\jump{s,t}$ and
\[\zeta^{s,t,\sh}_{\hcS}(\bk;T_{1},T_{2})-\zeta^{s,t,\ast}_{\hcS}(\bk;T_{1},T_{2})\in\zeta(2)\cZ[T_{1},T_{2}]\jump{s,t}.\]
In particular, $\zeta^{s,t}_{\hcS}(\bk)\coloneqq\zeta^{s,t,\bullet}_{\hcS}(\bk;T,T)\mod\zeta(2)$ is independent of $T$ and $\bullet$.
\end{Prop}
\begin{proof}
Using Proposition \ref{T_part}, we obtain
\begin{align}
\Phi^{\bullet}_{\Ad}(X_{0},X_{1};T_{1},T_{2})
&=\ep(\exp(-T_{1}X_{1})\Phi^{\bullet}(X_{0},X_{1};0))X_{1}\exp(-T_{2}X_{1})\Phi^{\bullet}(X_{0},X_{1};0)\\
&=\ep(\Phi^{\bullet}(X_{0},X_{1};0))X_{1}\exp((T_{1}-T_{2})X_{1})\Phi^{\bullet}(X_{0},X_{1};0)\\
&=\Phi^{\bullet}_{\Ad}(X_{0},X_{1};0,T_{2}-T_{1}).
\end{align}
Thus we obtain the first assertion by Proposition \ref{stadic_smzv_associator}.
For the rest part, it suffices to prove
\[\Phi^{\sh}_{\Ad}(X_{0},X_{1};0,T)\equiv\Phi^{\ast}_{\Ad}(X_{0},X_{1};0,T)\mod\zeta(2)\cZ[T]\jump{s,t}.\]
From Euler's theorem $\zeta(2k)\in\bbQ\zeta(2)^{k}$ and Proposition \ref{reg_thm_associator}, this is established as
\begin{align}
\Phi^{\sh}_{\Ad}(X_{0},X_{1};0,T)
&=\ep(\Phi^{\sh}(X_{0},X_{1};0))X_{1}\Phi^{\sh}(X_{0},X_{1};T)\\
&=\ep(\Phi^{\ast}(X_{0},X_{1};0))\Gamma_{0}(-X_{1})X_{1}\Gamma_{0}(X_{1})\Phi^{\ast}(X_{0},X_{1};T)\\
&=\ep(\Phi^{\ast}(X_{0},X_{1};0))X_{1}\exp\left(\sum_{k=1}^{\infty}\frac{\zeta(2k)}{k}X_{1}^{2k}\right)\Phi^{\ast}(X_{0},X_{1};T)\\
&\equiv\Phi^{\ast}_{\Ad}(X_{0},X_{1};0,T)\mod\zeta(2)\cZ[T]\jump{s,t}.\qedhere
\end{align}
\end{proof}
Therefore the value $\zeta_{\hcS}^{s,t}(\bk)$ is well-defined and coincides with the $t$-adic SMZV $\zeta_{\hcS}(\bk)$ when $s=0$. Note that, in Komori's notation \cite{komori21}, it is also written as $\zeta_{\hcS}^{s,t}(\bk)=\zeta_{\widehat{\mathcal{U}}}(\bk;-s,t)$ modulo $2\pi i\cZ[2\pi i]\jump{s,t}$.
\section{$(s,t)$-adic harmonic relation}\label{sec:harmonic_relation}
\begin{Lem}[Hoffman \cite{hoffman99}]\label{harmonic_hopf}
With the following data, $\fH^{1}$ becomes a commutative Hopf algebra over $\bbQ$.
\begin{enumerate}[\rm(1)]
\item The product is $\ast$.
\item The unit is the embedding map $\bbQ\ni a\mapsto a\in\fH^{1}$.
\item The coproduct is the $\bbQ$-linear map $\Delta\colon\fH^{1}\mapsto\fH^{1}\otimes\fH^{1}$ defined as
\[\Delta(e_{\bk})\coloneqq\sum_{i=0}^{\dep(\bk)}e_{\bk_{[i]}}\otimes e_{\bk^{[i]}}\qquad (\bk\in\cI).\]
\item The counit is the $\bbQ$-linear map $\fH^{1}\to\bbQ$ defined by $e_{(k_1,\dots,k_r)}\mapsto \delta_{r,0}$.
\item The antipode is the $\bbQ$-linear map $\fH^{1}\to\fH^{1}$ given by
\[e_{\bk}\mapsto\sum_{\bl\preceq\bk}(-1)^{\dep(\bl)}e_{\overleftarrow{\bl}}.\]
Here, for indices $\bk=(k_1,\dots,k_r)$ and $\bl$ with the same weight, the condition $\bl\preceq\bk$ is defined as $\bl\in\{(k_{1}\circ\cdots\circ k_{r})\mid\circ\text{ is a comma },\text{ or a plus }+\}$ and we understand $\emp\preceq\emp$.
\end{enumerate}
\end{Lem}
\begin{Thm}\label{general_stadic_harmonic}
For indices $\bk$ and $\bl$, we have
\[\zeta^{s,t,\ast}_{\hcS}(\bk\ast\bl;T_{1},T_{2})=\zeta^{s,t,\ast}_{\hcS}(\bk;T_{1},T_{2})\zeta^{s,t,\ast}_{\hcS}(\bl;T_{1},T_{2}).\]
\end{Thm}
\begin{proof}
The proof proceeds similarly to that of the second formula in \cite[Proposition 12]{hirose20}. In our case, defining $\bbQ$-linear maps $g^{s}_{T_{1}}\colon\fH^{1}\to\cZ[T_{1}]\jump{s}$ and $h^{t}_{T_{2}}\colon\fH^{1}\to\cZ[T_{2}]\jump{t}$ by
\[g^{s}_{T_{1}}(e_{\bk})\coloneqq (-1)^{\dep(\bk)}\zeta^{s,\ast}_{\shift}(\bk;T_{1}),\qquad h^{t}_{T_{2}}(e_{\bk})\coloneqq(-1)^{\wt(\bk)+\dep(\bk)}\zeta^{-t,\ast}_{\shift}(\bk;T_{2})\]
and $f\colon\cZ[T_{1}]\jump{s}\otimes\cZ[T_{2}]\jump{t}\to\cZ[T_{1},T_{2}]\jump{s,t}$ by $X\otimes Y\mapsto XY$, we see that the composite $f\circ(g^{s}_{T_{1}}\otimes h^{t}_{T_{2}})\circ\Delta$ is equal to $Z^{s,t,\ast}_{\hcS,T_{1},T_{2}}$.
The theorem follows from Lemma \ref{harmonic_hopf} and the fact that $g^{s}_{T_{1}}$ and $h^{t}_{T_{2}}$ are homomorphisms about $\ast$ (Corollary \ref{shifted_harmonic_relation}).
\end{proof}
Theorem \ref{stadic_harmonic_relation} immediately follows by putting $T_{1}=T_{2}=0$ in Theorem \ref{general_stadic_harmonic} and using Proposition \ref{stadic_independence}.
\section{$(s,t)$-adic shuffle relation}\label{sec:shuffle_relation}
Define the $s$-shifted shuffle product $\shuffle_{s}$ on $\fH\jump{s}$ by
\[
w\shuffle_{s}w'=(1-e_{0}s)\left(w\shuffle\frac{1}{1-e_{0}s}w'\right)
\]
where $w,w'\in\fH\jump{s}$. Note that $\fH^{1}\jump{s}$ is closed by $\shuffle_{s}$. Thus we can regard $\shuffle_{s}$ as the binary operator on $\cR\jump{s}$. The $s$-shifted product itself is not associative, but $(u\shuffle v)\shuffle_{s}w=u\shuffle_{s}(v\shuffle_{s}w)$ holds. By definition, we have the $s$-shifted shuffle product formula
\[\zeta_{\shift}^{s,\shuffle}(\bl\shuffle_{s}\bk)=\zeta^{\shuffle}(\bl)\zeta_{\shift}^{s,\shuffle}(\bk).\]
\begin{Thm}
We have
\[
\zeta_{\widehat{\mathcal{S}}}^{s,t,\shuffle}(\bl\shuffle_{s}\bk)=(-1)^{\wt(\bl)}\sum_{\bn\in\bbZ_{\geq0}^{\dep(\bl)}}b{\bl \choose \bn}\zeta_{\widehat{\mathcal{S}}}^{s,t,\shuffle}(\bk,\overleftarrow{\bl\oplus\bn})t^{\wt(\bn)}.
\]
\end{Thm}

\begin{proof}
Define $L:\fH\jump s\to\mathbb{R}\jump s$ by
\[
L(u)\coloneqq\left\langle \Phi_{\Ad}^{\shuffle}(X_{0},X_{1};0,0),u\right\rangle .
\]
By Proposition \ref{stadic_smzv_associator}, we have
\[
\zeta_{\hcS}^{s,t,\shuffle}(\bk)=(-1)^{\wt(\bk)+\dep(\bk)}L\left(\frac{1}{1+e_{0}s}e_{\bk}e_{1}\frac{1}{1+e_{0}t}\right).
\]
Since $\Phi_{\Ad}^{\shuffle}(X_{0},X_{1};0,0)$ is Lie-like, we have
\[
L(ue_{1}\epsilon(v))= L((u\shuffle v)e_{1}),
\]
where $\epsilon$ is the anti-automorphism defined by $\epsilon(e_{i})=-e_{i}$ for $i\in\{0,1\}$.
Thus,
\begin{align*}
 & (-1)^{\dep(\bk)+\dep(\bl)}\zeta_{\widehat{\mathcal{S}}}^{s,t,\shuffle}(\bl\shuffle_{s}\bk)\\
 & =(-1)^{\wt(\bk)+\wt(\bl)}L\left(\frac{1}{1+e_{0}s}(e_{\bl}\shuffle_{-s}e_{\bk})e_{1}\frac{1}{1+e_{0}t}\right)\\
 & =(-1)^{\wt(\bk)+\wt(\bl)}L\left(\left(e_{\bl}\shuffle\frac{1}{1+e_{0}s}e_{\bk}\right)e_{1}\frac{1}{1+e_{0}t}\right)\\
 & =(-1)^{\wt(\bk)+\wt(\bl)}L\left(\left(e_{\bl}\shuffle\frac{1}{1-e_{0}t}\shuffle\frac{1}{1+e_{0}s}e_{\bk}\right)e_{1}\right)\\
 & =(-1)^{\wt(\bk)+\wt(\bl)}L\left(\frac{1}{1+e_{0}s}e_{\bk}e_{1}\epsilon\left(e_{\bl}\shuffle\frac{1}{1-e_{0}t}\right)\right).
\end{align*}
Here,
\begin{align*}
e_{1}\epsilon\left(e_{\bl}\shuffle\frac{1}{1-e_{0}t}\right) & =e_{1}\epsilon\left(\frac{1}{1-e_{0}t}\sum_{\bn\in\bbZ_{\geq0}^{\dep(\bl)}}b{\bl \choose \bn}t^{\wt(\bn)}e_{\bl\oplus\bn}\right)\\
 & =\sum_{\bn\in\bbZ_{\geq0}^{\dep(\bl)}}b{\bl \choose \bn}t^{\wt(\bn)}e_{1}\epsilon(e_{\bl\oplus\bn})\frac{1}{1+e_{0}t}\\
 & =(-1)^{\wt(\bl)}\sum_{\bn\in\bbZ_{\geq0}^{\dep(\bl)}}b{\bl \choose \bn}
 (-t)^{\wt(\bn)}e_{\overleftarrow{\bl\oplus\bn}}e_{1}\frac{1}{1+e_{0}t},
\end{align*}
Thus,
\begin{align*}
(-1)^{\dep(\bk)+\dep(\bl)}\zeta_{\widehat{\mathcal{S}}}^{s,t,\shuffle}(\bl\shuffle_{s}\bk) & =\sum_{\bn\in\bbZ_{\geq0}^{\dep(\bl)}}(-1)^{\wt(\bk)}(-t)^{\wt(\bn)}b{\bl \choose \bn}L\left(\frac{1}{1+e_{0}s}e_{\bk}e_{\overleftarrow{\bl\oplus\bn}}e_{1}\frac{1}{1+e_{0}t}\right)\\
 & =\sum_{\bn\in\bbZ_{\geq0}^{\dep(\bl)}}(-1)^{\dep(\bk)+\dep(\bl)+\wt(\bl)}b{\bl \choose \bn}\zeta_{\widehat{\mathcal{S}}}^{s,t,\shuffle}(\bk,\overleftarrow{\bl\oplus\bn})t^{\wt(\bn)},
\end{align*}
which completes the proof.
\end{proof}

\section{$(s,t)$-adic duality}\label{sec:duality}
In this section, we introduce $(s,t)$-adic refined symmetric multiple zeta values in terms of the \emph{Knizhnik--Zamolodchikov associator} (=KZ associator) and prove the $(s,t)$-adic duality (Theorem \ref{stadic_duality}).
Note that $\Phi_{\KZ}(X_{0},X_{1})\coloneqq\Phi^{\sh}(X_{0},X_{1};0)$ is the KZ associator.
Drinfel'd \cite{drinfeld91} proved that the pair $(2\pi i,\Phi_{\KZ})$ is the Drinfel'd associator.
Especially, he proved the \emph{$2$-cycle relation} and \emph{$3$-cycle relation}.
\begin{Thm}[{\cite{drinfeld91}}]\label{23cycle}
Write $X_{\infty}\coloneqq -X_{0}-X_{1}$. Then we have the following.
\begin{description}
\item[$2$-cycle relation]\[\Phi_{\KZ}(X_{0},X_{1})\Phi_{\KZ}(X_{1},X_{0})=1.\]
\item[$3$-cycle relation]\[\Phi_{\KZ}(X_{0},X_{1})\exp(\pi iX_{0})\Phi_{\KZ}(X_{\infty},X_{0})\exp(\pi iX_{\infty})\Phi_{\KZ}(X_{1},X_{\infty})\exp(\pi iX_{1})=1.\]
\end{description}
\end{Thm}
We define
\[\Phi_{\cRS}(X_{0},X_{1})\coloneqq\exp(\pi iX_{0}/2)\Phi_{\KZ}(X_1,X_0)\exp(2\pi iX_{1})\Phi_{\KZ}(X_0,X_1)\exp(\pi iX_{0}/2)\]
and
\[\zeta^{s,t}_{\hcRS}(\bk)\coloneqq\frac{(-1)^{\wt(\bk)+\dep(\bk)}}{2\pi i}\left\langle\Phi_{\cRS}(X_{0},X_{1}),~\frac{1}{1+e_{0}s}e_{1}\omega(\bk)e_{1}\frac{1}{1+e_{0}t}\right\rangle,\]
for a non-empty index $\bk$. We put $\zeta^{s,t}_{\hcRS}(\emp)\coloneqq\exp(-(s+t)\pi i/2)$.
\begin{Rem}
By an argument similar to the latter part of Proposition \ref{stadic_independence},
it holds that
\begin{align}
&\Phi^{\ast}_{\Ad}(X_{0},X_{1};\pi i/2,-\pi i/2)\\
&=\ep(\exp(-\pi iX_{1}/2)\Phi^{\ast}(X_{0},X_{1};0))X_{1}\exp(\pi iX_{1}/2)\Phi^{\ast}(X_{0},X_{1};0)\\
&=\ep(\Phi^{\sh}(X_{0},X_{1};0))\Gamma_{0}(-X_{1})^{-1}\exp(\pi iX_{1})X_{1}\Gamma_{0}(X_{1})^{-1}\Phi^{\sh}(X_{0},X_{1};0)\\
&=\ep(\Phi^{\sh}(X_{0},X_{1};0))\frac{\exp(2\pi iX_{1})-1}{2\pi i}\Phi^{\sh}(X_{0},X_{1};0)\\
&= \frac{1}{2\pi i} \exp(-\pi i X_{0}/2)\Phi_{\cRS}(X_{0},X_{1})\exp(-\pi iX_{0}/2)-\frac{1}{2\pi i}.
\end{align}
Hence we have
\begin{align}
&\zeta^{s,t}_{\hcRS}(\bk)\\
&= (-1)^{\wt(\bk)+\dep(\bk)}  \left\langle \exp(\pi iX_{0}/2)\Phi^{\ast}_{\Ad}(X_{0},X_{1};\pi i/2,-\pi i/2)\exp(\pi iX_{0}/2)+\frac{\exp(\pi iX_{0})}{2\pi i} ,~\frac{1}{1+e_{0}s}e_{1}\omega(\bk)e_{1}\frac{1}{1+e_{0}t}\right\rangle\\
&= (-1)^{\wt(\bk)+\dep(\bk)} \sum_{m,n=0}^{\infty}\biggl(\sum_{\substack{0\le a\le m\\ 0\le b\le n}} \frac{\left(\pi i/2\right)^{m-a+n-b}}{(m-a)!(n-b)!}\langle\Phi^{\ast}_{\Ad}(X_{0},X_{1};\pi i/2,-\pi i/2),~e_{0}^{a}e_{1}\omega(\bk)e_{1}e_{0}^{b}\rangle\biggr)(-s)^{m}(-t)^{n}\\
&=\exp(-(s+t)\pi i/2)\zeta^{s,t,\ast}_{\hcS}(\bk;\pi i/2,-\pi i/2).
\end{align}
This equality certifies that the harmonic relation for $(s,t)$-adic RSMZVs
\[\exp(-(s+t)\pi i/2)\zeta^{s,t}_{\hcRS}(\bk\ast\bl)=\zeta^{s,t}_{\hcRS}(\bk)\zeta^{s,t}_{\hcRS}(\bl)\]
holds for indices $\bk$ and $\bl$ via Theorem \ref{general_stadic_harmonic}. Furthermore, combining this result with the last assertion of \cite[Corollary 12]{hirose20}, we find that $\zeta_{\hcRS}^{0,0}(\bk)\in \cZ[2\pi i]$ is the complex conjugate of $\zeta^{RS}(\bk)$ in \cite{hirose20}.
\end{Rem}
\begin{Prop}\label{lifting_property}
For an index $\bk$, we have
\[\zeta^{s,t,\sh}_{\hcS}(\bk)\equiv\zeta^{s,t}_{\hcRS}(\bk)\mod 2\pi i\cZ[2\pi i]\jump{s,t}.\]
\end{Prop}
\begin{proof}
By the $2$-cycle relation, we get
\begin{align}
\Phi_{\cRS}(X_{0},X_{1})
&=\exp(\pi iX_{0}/2)\Phi_{\KZ}(X_{1},X_{0})\exp(2\pi iX_{1})\Phi_{\KZ}(X_{0},X_{1})\exp(\pi iX_{0}/2)\\
&=\exp(\pi iX_{0}/2)\Phi_{\KZ}(X_{1},X_{0}) \left( 1+2\pi iX_{1} + \frac{(2\pi iX_{1})^2}{2}+\cdots  \right)  \Phi_{\KZ}(X_{0},X_{1})\exp(\pi iX_{0}/2)\\
&\equiv \exp(\pi iX_{0}) + 2\pi i\Phi^{\sh}_{\Ad}(X_{0},X_{1};0)  \mod (2\pi i)^2\cZ[2\pi i]\ncjump{X_{0},X_{1}}.
\end{align}
The proposition is a consequence of this congruence.
\end{proof}
\begin{Prop}\label{duality_associator}
We have
\[\Phi_{\cRS}(X_{\infty},X_0)=\overline{\Phi_{\cRS}(X_{\infty},X_1)},\]
where $\overline{\Phi}$ denotes the complex conjugate of $\Phi$ (acting coefficientwise).
\end{Prop}
\begin{proof}
From Theorem \ref{23cycle}, we see that
\[\Phi_{\KZ}(X_{\infty},X_{0})=\exp(-\pi iX_{0})\Phi_{\KZ}(X_{1},X_{0})\exp(-\pi iX_{1})\Phi_{\KZ}(X_{\infty},X_{1})\exp(-\pi iX_{\infty})\]
holds.
Furthermore, substituting $X_{1}$ and $X_{\infty}$ in Theorem \ref{23cycle}, we have
\[\Phi_{\KZ}(X_{0},X_{\infty})\exp(\pi iX_{0})\Phi_{\KZ}(X_{1},X_{0})=\exp(-\pi iX_{\infty})\Phi_{\KZ}(X_{1},X_{\infty})\exp(-\pi iX_{1}).\]
Hence we obtain
\begin{align}
&\Phi_{\cRS}(X_{\infty},X_{0})
\\&=\exp(\pi iX_{\infty}/2)\Phi_{\KZ}(X_{0},X_{\infty})\exp(\pi iX_{0})\Phi_{\KZ}(X_{1},X_{0})\exp(-\pi iX_{1})\Phi_{\KZ}(X_{\infty},X_{1})\exp(-\pi iX_{\infty}/2)\\
&=\exp(-\pi iX_{\infty}/2)\Phi_{\KZ}(X_{1},X_{\infty})\exp(-2\pi iX_{1})\Phi_{\KZ}(X_{\infty},X_{1})\exp(-\pi iX_{\infty}/2)\\
&=\overline{\Phi_{\cRS}(X_{\infty},X_{1})}.\qedhere
\end{align}
\end{proof}
The $(s,t)$-adic (refined) symmetric multiple zeta star value is defined by
\[\zeta^{s,t,\star}_{?}(\bk)\coloneqq\sum_{\bl\preceq\bk}\zeta^{s,t}_{?}(\bl)\]
for a non-empty index $\bk=(k_{1},\ldots,k_{r})$ and $?\in\{\hcS,\hcRS\}$.
Then we easily obtain
\[\zeta^{s,t,\star}_{\hcS}(\bk)\equiv\zeta^{s,t,\star}_{\hcRS}(\bk)\mod 2\pi i\cZ[2\pi i]\jump{s,t}\]
from Proposition \ref{lifting_property}.
\begin{Prop}\label{stadic_smzsv_associator}
For a non-empty index $\bk$, we have
\[\zeta^{s,t,\star}_{\hcRS}(\bk)=\frac{1}{2\pi i}\left\langle\Phi_{\cRS}(X_{\infty},X_{1}),~\frac{1}{1-e_{0}s}(e_{1}-e_{0})\omega(\bk)(e_{1}-e_{0})\frac{1}{1-e_{0}t}\right\rangle.\]
\end{Prop}
\begin{proof}
For non-negative integers $m$ and $n$, we denote by $\zeta_{\hcRS}(\bk;m,n)$ (resp. $\zeta^{\star}_{\hcRS}(\bk;m,n)$) the coefficient of $\zeta^{s,t}_{\hcRS}(\bk)$ (resp. $\zeta^{s,t,\star}_{\hcRS}(\bk)$) at $s^{m}t^{n}$.
By definition, there exists $\psi(X_{0},X_{1})\in \bbC\jump{X_{0}}+\bbC\jump{X_{0}}X_1\bbC\jump{X_{0}}$ such that
\[\Phi_{\cRS}(X_{0},X_{1}) - \psi(X_{0},X_{1})
=2\pi i\sum_{\bk\in\cI\setminus\{\emp\}}\sum_{\substack{0\le a\\ 0\le b}}(-1)^{\wt(\bk)+\dep(\bk)+a+b}\zeta_{\hcRS}(\bk;a,b)X_{0}^{a}X_{1}\omega(\bk)X_{1}X_{0}^{b}.\]
Let $\tilde{\Phi}_{\cRS}(X_{0},X_{1})$ denote the right-hand side. Then we have
\begin{align}
\tilde{\Phi}_{\cRS}(X_{\infty},X_{1})
&=\begin{multlined}[t]
2\pi i\sum_{\substack{\bk\in\cI\setminus\{\emp\}\\ \bk=(k_{1},\ldots,k_{r})}}\sum_{\substack{0\le a\\ 0\le b}}(-1)^{\wt(\bk)+\dep(\bk)}\zeta_{\hcRS}(\bk;a,b)\\
\cdot(X_{0}+X_{1})^{a}X_{1}(-X_{0}-X_{1})^{k_{1}-1}\cdots X_{1}(-X_{0}-X_{1})^{k_{r}-1}X_{1}(X_{0}+X_{1})^{b}
\end{multlined}\\
&=2\pi i\sum_{\bk\in\cI\setminus\{\emp\}}\sum_{\substack{0\le a\\ 0\le b}}\zeta_{\hcRS}(\bk;a,b)\sum_{\bk\preceq\bl}(X_{0}+X_{1})^{a}X_{1}\omega(\bl)X_{1}(X_{0}+X_{1})^{b}\\
&=2\pi i\sum_{\bl\in\cI\setminus\{\emp\}}\sum_{\substack{0\le a\\ 0\le b}}\left(\sum_{\bk\preceq\bl}\zeta_{\hcRS}(\bk;a,b)\right)(X_{0}+X_{1})^{a}X_{1}\omega(\bl)X_{1}(X_{0}+X_{1})^{b}.
\end{align}
Therefore we obtain
\[
\Phi_{\cRS}(X_{\infty},X_{1})-\psi(X_{\infty},X_{1})
=2\pi i\sum_{\bk\in\cI\setminus\{\emp\}}\sum_{\substack{0\le a\\ 0\le b}}\zeta^{\star}_{\hcRS}(\bk;a,b)(X_{0}+X_{1})^{a}X_{1}\omega(\bk)X_{1}(X_{0}+X_{1})^{b}.
\]
Since $\langle X_{0}+X_{1},e_{0}\rangle=1$, $\langle X_{0}+X_{1},e_{1}-e_{0}\rangle=0$, $\langle X_{1},e_{0}\rangle=0$ and $\langle X_1,e_{1}-e_{0}\rangle=1$, we have
\[
\langle (X_{0}+X_{1})^{a}X_{1}\omega(\bk')X_{1}(X_{0}+X_{1})^{b}, ~\frac{1}{1-e_{0}s}(e_{1}-e_{0})\omega(\bk)(e_{1}-e_{0})\frac{1}{1-e_{0}t} \rangle = s^{a}t^{b}\delta_{\bk',\bk}
\]
and
\[
\langle \psi(X_{\infty},X_{1}), ~\frac{1}{1-e_{0}s}(e_{1}-e_{0})\omega(\bk)(e_{1}-e_{0})\frac{1}{1-e_{0}t}\rangle = 0.
\]
Thus
\[
\frac{1}{2\pi i}\left\langle\Phi_{\cRS}(X_{\infty},X_{1}), ~\frac{1}{1-e_{0}s}(e_{1}-e_{0})\omega(\bk)(e_{1}-e_{0})\frac{1}{1-e_{0}t}\right\rangle
= \sum_{\substack{0\le a\\ 0\le b}}\zeta_{\hcRS}(\bk;a,b)s^{a}t^{b} = \zeta^{s,t}_{\hcRS}(\bk).\qedhere
\]
\end{proof}
When we express a non-empty index $\bk=(k_{1},\ldots,k_{r})$ as
\[\bk=(\underbrace{1+\cdots+1}_{k_{1}},\ldots,\underbrace{1+\cdots+1}_{k_{r}}),\]
we denote by $\bk^{\vee}$ the index obtained by interchanging commas `$,$' and pluses `$+$'.
\begin{Thm}\label{refined_stadic_duality}
For a non-empty index $\bk$, we have
\[\sum_{m,n=0}^{\infty} \zeta^{s,t,\star}_{\hcRS}(\{1\}^m,\bk,\{1\}^n)s^mt^n=-\sum_{m,n=0}^{\infty} \overline{\zeta^{s,t,\star}_{\hcRS}(\{1\}^m,\bk^{\vee},\{1\}^n)}s^mt^n.\]
\end{Thm}
\begin{proof}
Let $\tau$ be the automorphism of $\bbQ\langle e_{0},e_{1}\rangle$ interchanging $e_{0}$ and $e_{1}$ and extend it to $\bbQ\langle e_{0},e_{1}\rangle\jump{s,t}$ by putting $\tau(s)=s$ and $\tau(t)=t$. Then Proposition \ref{duality_associator} implies that
\[\langle\Phi_{\cRS}(X_{\infty},X_{1}),~w\rangle=\langle\Phi_{\cRS}(X_{\infty},X_{0}),~\tau(w)\rangle=\overline{\langle\Phi_{\cRS}(X_{\infty},X_{1}),~\tau(w)\rangle}\]
for $w\in\bbQ\langle e_{0},e_{1}\rangle\jump{s,t}$. This equality and Proposition \ref{stadic_smzsv_associator} shows that
\begin{align}
&\sum_{m,n=0}^{\infty} \zeta^{s,t,\star}_{\hcRS}(\{1\}^m,\bk,\{1\}^n)s^mt^n\\
&=\frac{1}{2\pi i}\left\langle\Phi_{\cRS}(X_{\infty},X_{1}),~\frac{1}{1-e_{0}s}(e_{1}-e_{0})\frac{1}{1-e_{1}s}\omega(\bk)\frac{1}{1-e_{1}t}(e_{1}-e_{0})\frac{1}{1-e_{0}t}\right\rangle\\
&=\frac{1}{2\pi i}\overline{\left\langle\Phi_{\cRS}(X_{\infty},X_{1}),~\tau\left(\frac{1}{1-e_{0}s}(e_{1}-e_{0})\frac{1}{1-e_{1}s}\omega(\bk)\frac{1}{1-e_{1}t}(e_{1}-e_{0})\frac{1}{1-e_{0}t}\right)\right\rangle}
\end{align}
for a non-empty index $\bk$. By the definition of $\tau$, we have
\begin{multline}
\tau\left(\frac{1}{1-e_{0}s}(e_{1}-e_{0})\frac{1}{1-e_{1}s}\omega(\bk)\frac{1}{1-e_{1}t}(e_{1}-e_{0})\frac{1}{1-e_{0}t}\right)\\
=\frac{1}{1-e_{0}s}(e_{1}-e_{0})\frac{1}{1-e_{1}s}\omega(\bk^{\vee})\frac{1}{1-e_{1}t}(e_{1}-e_{0})\frac{1}{1-e_{0}t}
\end{multline}
since $\frac{1}{1-e_{0}a}(e_{1}-e_{0})\frac{1}{1-e_{1}a}=\frac{1}{a}(\frac{1}{1-e_1a}-\frac{1}{1-e_0a})$ is multiplied by $-1$ by the action of $\tau$.
Therefore we obtain
\begin{align}
&\frac{1}{2\pi i}\overline{\left\langle\Phi_{\cRS}(X_{\infty},X_{1}),~\tau\left(\frac{1}{1-e_{0}s}(e_{1}-e_{0})\frac{1}{1-e_{1}s}\omega(\bk)\frac{1}{1-e_{1}t}(e_{1}-e_{0})\frac{1}{1-e_{0}t}\right)\right\rangle}\\
&=-\overline{\frac{1}{2\pi i}\left\langle\Phi_{\cRS}(X_{\infty},X_{1}),~\frac{1}{1-e_{0}s}(e_{1}-e_{0})\frac{1}{1-e_{1}s}\omega(\bk^{\vee})\frac{1}{1-e_{1}t}(e_{1}-e_{0})\frac{1}{1-e_{0}t}\right\rangle}\\
&=-\sum_{m,n=0}^{\infty}\overline{\zeta^{s,t,\star}_{\hcRS}(\{1\}^{m},\bk^{\vee},\{1\}^{n})}s^{m}t^{n}.\qedhere
\end{align}
\end{proof}
Taking modulo $2\pi i$ in Theorem \ref{refined_stadic_duality}, by Proposition \ref{lifting_property}, we get Theorem \ref{stadic_duality}. Moreover, we recover the $t$-adic duality obtained by Takeyama--Tasaka \cite[Corollary 5.5]{tt23} by putting $s=0$ in Theorem \ref{stadic_duality}.
\section{$(s,t)$-adic cyclic sum formula}\label{sec:cyclic_sum_formula}
\subsection{Proof of the $(s,t)$-adic star cyclic sum formula}\label{subset:cyclic_sum_formula_star}
For an index $\bk$, put
\[\zeta^{\star,\bullet}(\bk;T)\coloneqq\sum_{\bl\preceq\bk}\zeta^{\bullet}(\bl;T),\]

\[\zeta^{t,\star,\bullet}_{\shift}(\bk;T)\coloneqq\sum_{\be\in\bbZ_{\ge 0}^{\dep(\bk)}}b\left(\bk\atop\be\right)\zeta^{\star,\bullet}(\bk\oplus\be;T)(-t)^{\wt(\be)}\ \ \ \left(=\sum_{\bl\preceq\bk}\zeta^{t,\bullet}_{\shift}(\bl;T)\right),\]
and
\[
\zeta^{s,t,\star,\bullet}_{\hcS}(\bk;T_{1},T_{2}) \coloneqq \sum_{i=0}^{\dep(\bk)}(-1)^{\wt(\bk^{[i]})}\zeta^{s,\star,\bullet}_{\shift}(\bk_{[i]};T_{1})\zeta^{-t,\star,\bullet}_{\shift}(\overleftarrow{\bk^{[i]}};T_{2})\ \ \ \left(=\sum_{\bl\preceq\bk}\zeta^{s,t,\bullet}_{\hcS}(\bl;T_{1},T_{2})\right).
\]
We easily see that $\zeta^{s,t,\star}_{\hcS}(\bk)=\zeta^{s,t,\star,\bullet}_{\hcS}(\bk;0,0)$ holds in $(\cZ/\zeta(2)\cZ)\jump{s,t}$, independently of $\bullet$, according to Proposition \ref{stadic_independence}.
This subsection is devoted to the proof of the following formula:
\begin{Thm}\label{general_stadic_csf}
For an index $\bk=(k_{1},\ldots,k_{r})$ whose weight $k$ is greater than $r$, we have
\begin{align}
&\sum_{i=1}^{r}\sum_{j=0}^{k_i-1}\zeta^{s,t,\star,\bullet}_{\hcS}(j+1,\bk^{[i]},\bk_{[i-1]},k_{i}-j;T_{1},T_{2})\\
&=\sum_{i=1}^{r}\sum_{j=0}^{\infty}\left(\zeta^{s,t,\star,\bullet}_{\hcS}(j+1,\bk^{[i]},\bk_{[i]};T_{1},T_{2})t^{j}+\zeta^{s,t,\star,\bullet}_{\hcS}(\bk^{[i]},\bk_{[i]},j+1;T_{1},T_{2})s^j\right)+k\zeta^{s,t,\star,\bullet}_{\hcS}(k+1;T_{1},T_{2}).
\end{align}
\end{Thm}
In this section, we often use the following notation of the cyclic equivalent classes: For an index $\bk$ with depth $r$, write
\[\alpha(\bk)\coloneqq\{\!|~(\bk^{[i]},\bk_{[i]})\mid 1\le i\le r~|\!\}\]
as a multiset. A sum over a multiset counts the multiplicity. For example, the statement of Theorem \ref{general_stadic_csf} is written as
\begin{align}
&\sum_{(u,\bl)\in\alpha(\bk)}\sum_{j=0}^{u-1}\zeta^{s,t,\star,\bullet}_{\hcS}(j+1,\bl,u-j;T_{1},T_{2})\\
&=\sum_{\bl\in\alpha(\bk)}\sum_{j=0}^{\infty}\left(\zeta^{s,t,\star,\bullet}_{\hcS}(j+1,\bl;T_{1},T_{2})t^{j}+\zeta^{s,t,\star,\bullet}_{\hcS}(\bl,j+1;T_{1},T_{2})s^{j}\right)+k\zeta^{s,t,\star,\bullet}_{\hcS}(k+1;T_{1},T_{2}),
\end{align}
where $\bl$ shown in bold font (resp.~$u$ in fine font) denotes an index (resp.~a positive integer, single index) and we use such notations likewise hereafter.
The left-hand side is decomposed as
\begin{equation}\label{eq:decomposition_stadic_csf}
\begin{split}
&\sum_{(u,\bl)\in\alpha(\bk)}\sum_{j=0}^{u-1}\zeta^{s,t,\star,\bullet}_{\hcS}(j+1,\bl,u-j;T_{1},T_{2})\\
&=\sum_{(u,\bl)\in\alpha(\bk)}\sum_{j=0}^{u-1}\left(\zeta^{s,\star,\bullet}_{\shift}(j+1,\bl,u-j;T_{1})+(-1)^{k+1}\zeta^{-t,\star,\bullet}_{\shift}(u-j,\overleftarrow{\bl},j+1;T_{2})\right)\\
&\qquad+\sum_{(u,\bl)\in\alpha(\bk)}\sum_{j=0}^{u-1}\sum_{(\bm,\bn)=\bl}(-1)^{\wt(\bn)+u-j}\zeta^{s,\star,\bullet}_{\shift}(j+1,\bm;T_{1})\zeta^{-t,\star,\bullet}_{\shift}(u-j,\overleftarrow{\bn};T_{2}).
\end{split}
\end{equation}
To compute the first sum, we prove the cyclic sum formula for shifted multiple zeta values.
\begin{Thm}[{\cite[Theorem 1]{ow06}}]\label{classical_csf}
For an index $\bk$ with weight $k$ greater than $\dep(\bk)$, we have
\[\sum_{(u,\bl)\in\alpha(\bk)}\sum_{j=0}^{u-2}\zeta^{\star}(j+1,\bl,u-j)=k\zeta^{\star}(
k+1),\]
where $\zeta^{\star}(\bk)\coloneqq\sum_{\bl\preceq\bk}\zeta(\bl)$ is the usual multiple zeta star value and the inner sum on the left-hand side is regarded as $0$ when $u=1$.
\end{Thm}
\begin{Prop}[Cyclic sum formula for shifted MZVs]\label{shifted_csf}
For an index $\bk$ with weight $k$ greater than depth $r$, we have
\[\sum_{(u,\bl)\in\alpha(\bk)}\sum_{j=0}^{u-1}\zeta^{t,\star,\bullet}_{\shift}(j+1,\bl,u-j;T)=\sum_{\bl\in\alpha(\bk)}\sum_{j=0}^{\infty}\zeta^{t,\star,\bullet}_{\shift}(\bl,j+1;T)t^{j}+k\zeta^{t,\star,\bullet}_{\shift}(k+1;T).\]
\end{Prop}
\begin{proof}
The left-hand side is equal to
\begin{align}
&\sum_{(u,\bl)\in\alpha(\bk)}\sum_{\substack{j_{1},j_{2}\ge 0\\ j_{1}+j_{2}=u-1}}\zeta^{t,\star,\bullet}_{\shift}(j_{1}+1,\bl,j_{2}+1;T)\\
&=\begin{multlined}[t]
\sum_{(u,\bl)\in\alpha(\bk)}\sum_{\substack{j_{1},j_{2}\ge 0\\ j_{1}+j_{2}=u-1}}\sum_{\substack{\bn=(n_{1},\ldots,n_{r-1})\in\bbZ_{\ge 0}^{r-1}\\ m,n\ge 0}}b\left({\bl\atop\bn}\right)\binom{m+j_{1}}{m}\binom{n+j_{2}}{n}\\
\cdot\zeta^{\star,\bullet}(j_{1}+m+1,\bl\oplus\bn,j_{2}+n+1;T)(-t)^{\wt(\be)+m+n}
\end{multlined}\\
&=\begin{multlined}[t]
\sum_{(u,\bl)\in\alpha(\bk)}\sum_{\substack{\bn=(n_{1},\ldots,n_{r-1})\in\bbZ_{\ge 0}^{r-1}\\ N\ge 0}}\sum_{\substack{h_{1},h_{2}\ge 0\\ h_{1}+h_{2}=N+u-1}}\left(\sum_{\substack{m,n\ge 0\\ m+n=N}}\binom{h_{1}}{m}\binom{h_{2}}{n}\right)b\left({\bl\atop\bn}\right)\\
\cdot\zeta^{\star,\bullet}(h_{1}+1,\bl\oplus\bn,h_{2}+1;T)(-t)^{\wt(\be)+N}.
\end{multlined}
\end{align}
Here, the expression enclosed by the big parentheses coincides with $\binom{h_{1}+h_{2}}{N}$ by easy computation. Then we obtain
\begin{align}
&\sum_{(u,\bl)\in\alpha(\bk)}\sum_{\substack{j_{1},j_{2}\ge 0\\ j_{1}+j_{2}=u-1}}\zeta^{t,\star,\bullet}_{\shift}(j_{1}+1,\bl,j_{2}+1;T)\\
&=\sum_{(u,\bl)\in\alpha(\bk)}\sum_{n_{1},\ldots,n_{r-1},N\ge 0}b\left({\bl\atop{n_{1},\ldots,n_{r-1}}}\right)\sum_{\substack{h_{1},h_{2}\ge 0\\ h_{1}+h_{2}=N+u-1}}\binom{h_{1}+h_{2}}{N}\zeta^{\star,\bullet}(h_{1}+1,\bl\oplus\bn,h_{2}+1;T)(-t)^{\wt(\be)+N}\\
&=\sum_{\bn\in\bbZ_{\ge 0}^{r}}b\left({\bk\atop\bn}\right)(-t)^{\wt(\bn)}\sum_{(v,\bm)\in\alpha(\bk\oplus\bn)}\sum_{\substack{h_{1},h_{2}\ge 0\\ h_{1}+h_{2}=v-1}}\zeta^{\star,\bullet}(h_{1}+1,\bm,h_{2}+1;T).
\end{align}
The inner sum $\sum_{(v,\bm)}\sum_{h_{1},h_{2}}$ can be computed by the ordinary cyclic sum formula (Theorem \ref{classical_csf}) as
\[\sum_{(v,\bm)\in\alpha(\bk\oplus\bn)}\sum_{\substack{h_{1},h_{2}\ge 0\\ h_{1}+h_{2}=v-1}}\zeta^{\star,\bullet}(h_{1}+1,\bm,h_{2}+1;T)=\sum_{(v,\bm)\in\alpha(\bk\oplus\bn)}\zeta^{\star,\bullet}(v,\bm,1;T)+(k+\wt(\bn))\zeta(k+\wt(\bn)+1),\]
and therefore we have
\begin{align}
&\sum_{(u,\bl)\in\alpha(\bk)}\sum_{\substack{j_{1},j_{2}\ge 0\\ j_{1}+j_{2}=u-1}}\zeta^{t,\star,\bullet}_{\shift}(j_{1}+1,\bl,j_{2}+1;T)\\
&=\sum_{\bn\in\bbZ_{\ge 0}^{r}}b\left({\bk\atop\bn}\right)(-t)^{\wt(\bn)}\left(\sum_{\bl\in\alpha(\bk\oplus\bn)}\zeta^{\star,\bullet}(\bl,1;T)+(k+\wt(\bn))\zeta(k+\wt(\bn)+1)\right).
\end{align}
The second term on the right-hand side is $k\zeta_{\shift}^{t}(k+1)$ from a simple calculation of binomial coefficients, whereas the first term is
\begin{align}
&\sum_{\bn\in\bbZ_{\ge 0}^{r}}b\left({\bk\atop\bn}\right)\sum_{\bl\in\alpha(\bk\oplus\bn)}\zeta^{\star,\bullet}(\bl,1;T)(-t)^{\wt(\bn)}\\
&=\sum_{\bn\in\bbZ_{\ge 0}^{r}}b\left({\bk\atop\bn}\right)\sum_{\bl\in\alpha(\bk\oplus\bn)}\sum_{n=0}^{\infty}\left(\sum_{i=0}^{n}(-1)^{n-i}\binom{n}{i}\right)\zeta^{\star,\bullet}(\bl,n+1;T)(-t)^{n+\wt(\bn)}\\
&=\sum_{\bn\in\bbZ_{\ge 0}^{r}}b\left({\bk\atop\bn}\right)\sum_{\bl\in\alpha(\bk\oplus\bn)}\sum_{j,n=0}^{\infty}(-1)^{j}\binom{j+n}{n}\zeta^{\star,\bullet}(\bl,j+n+1;T)(-t)^{j+n+\wt(\bn)}\\
&=\sum_{\bl\in\alpha(\bk)}\sum_{j=0}^{\infty}\zeta^{t,\star,\bullet}_{\shift}(\bl,j+1;T)t^{j}.
\end{align}
This finishes the proof.
\end{proof}
By this proposition, the first sum of \eqref{eq:decomposition_stadic_csf} can be expressed as
\begin{align}
&\sum_{(u,\bl)\in\alpha(\bk)}\sum_{j=0}^{u-1}\left(\zeta^{s,\star,\bullet}_{\shift}(j+1,\bl,u-j;T_{1})+(-1)^{k+1}\zeta^{-t,\star,\bullet}_{\shift}(u-j,\overleftarrow{\bl},j+1;T_{2})\right)\\
&=k\zeta^{s,\star,\bullet}_{\shift}(k+1)+(-1)^{k+1}k\zeta^{t,\star,\bullet}_{\shift}(k+1)+\sum_{\bl\in\alpha(\bk)}\sum_{j=0}^{\infty}(\zeta^{s,\star,\bullet}_{\shift}(\bl,j+1;T_{1})s^{j}+(-1)^{k+j+1}\zeta^{-t,\star,\bullet}_{\shift}(\overleftarrow{\bl},j+1;T_{2})t^{j})\\
&=k\zeta^{s,t,\star,\bullet}_{\hcS}(k+1;T_{1},T_{2})+\sum_{\bl\in\alpha(\bk)}\sum_{j=0}^{\infty}(\zeta^{s,\star,\bullet}_{\shift}(\bl,j+1;T_{1})s^{j}+(-1)^{k+j+1}\zeta^{-t,\star,\bullet}_{\shift}(\overleftarrow{\bl},j+1;T_{2})t^{j}).
\end{align}
We denote by $F^{\bullet}(\bk)$ the remaining term of \eqref{eq:decomposition_stadic_csf}. Then it is sufficient to prove that
\begin{multline}\label{eq:definition_of_F}
F^{\bullet}(\bk)=\sum_{\bl\in\alpha(\bk)}\sum_{j=0}^{\infty}(\zeta^{s,t,\star,\bullet}_{\hcS}(j+1,\bl;T_{1},T_{2})-(-1)^{k+j+1}\zeta^{-t,\star,\bullet}_{\shift}(\overleftarrow{\bl},j+1;T_{2}))t^{j}\\
+\sum_{\bl\in\alpha(\bk)}\sum_{j=0}^{\infty}\left(\zeta^{s,t,\star,\bullet}_{\hcS}(\bl,j+1;T_{1},T_{2})-\zeta^{s,\star,\bullet}_{\shift}(\bl,j+1;T_{1})\right)s^{j}.
\end{multline}
Moreover, from the regularization theorem (Theorem \ref{ikz_reg_thm}) for $T_{1}$ and $T_{2}$ in the above equality, it suffices to show the case where $\bullet=\ast$. On the other hand, the definition of $F^{\ast}(\bk)$ yields the following:
\begin{align}
F^{\ast}(\bk)
&=\sum_{(u,\bl)\in\alpha(\bk)}\sum_{j=0}^{u-1}\sum_{(\bm,\bn)=\bl}(-1)^{\wt(\bn)+u-j}\zeta^{s,\star,\ast}_{\shift}(j+1,\bm;T_{1})\zeta^{-t,\star,\ast}_{\shift}(u-j,\overleftarrow{\bn};T_{2})\\
&=\sum_{\bl\in\alpha(\bk)}\sum_{(\bn,u,\bm)=\bl}\sum_{j=0}^{u-1}(-1)^{\wt(\bn)+u-j}\zeta^{s,\star,\ast}_{\shift}(j+1,\bm;T_{1})\zeta^{-t,\star,\ast}_{\shift}(u-j,\overleftarrow{\bn};T_{2}).
\end{align}
We denote by $f(\bl)$ the summand of $\sum_{\bl\in\alpha(\bk)}$; that is,
\[f(\bl)\coloneqq\sum_{(\bn,u,\bm)=\bl}\sum_{j=0}^{u-1}(-1)^{\wt(\bn)+u-j}\zeta^{s,\star,\ast}_{\shift}(j+1,\bm;T_{1})\zeta^{-t,\star,\ast}_{\shift}(u-j,\overleftarrow{\bn};T_{2}).\]
Furthermore, for an index $\bk=(k_{1},\ldots,k_{r})$, we write
\[R(\bk;T)\coloneqq\left(\prod_{i=1}^{r}\delta_{k_{i},1}\right)\sum_{\substack{a+b=r\\a,b\geq 0}}\frac{\rho(T^{b})\mid_{T=0}(-T)^a}{a!b!}. \]
Then the regularization theorem shows
\begin{equation}\label{eq:explicit_reg_thm}
\zeta^{t,\sh}_{\shift}(\bk;0)=\sum_{(\bk',\bk'')=\bk}\zeta^{t,\ast}_{\shift}(\bk';T)R(\bk'';T)
\end{equation}
since we have
\[
\zeta^{t,\sh}_{\shift}(\bk;0) = \sum_{ \substack{\bk = (\bk',\{1\}^b) \\ \bk'\in\cI,~b\geq 0} }\zeta^{t,\ast}_{\shift}(\bk';0)\frac{\rho(T^b)\mid_{T=0}}{b!}=  \sum_{\substack{\bk = (\bk',\{1\}^{a+b})\\ \bk'\in\cI,~a,b\geq 0 }   }\zeta^{t,\ast}_{\shift}(\bk';T) \frac{\rho(T^b)\mid_{T=0}(-T)^{a}}{a!b!}.
\]

\begin{Lem}\label{f_decomposition}
We have
\[f(\bl)=f_{1}(\bl)+f_{2}(\bl)+f_{3}(\bl)+f_{4}(\bl)\]
for any index $\bl$ with $\wt(\bl)>\dep(\bl)$, where we write
\begin{align}
f^{s,t}_{1}(\bl;T_{1},T_{2})&\coloneqq
\begin{multlined}[t]
\sum_{(\bn'',\bn',u,\bm',\bm'')=\bl}\sum_{j=0}^{u-1}(-1)^{\wt(\bn'',\bn')+\dep(\bn',\bm')+u-j}\\
\cdot\zeta^{s,\star,\ast}_{\shift}(\bm'';T_{1})\zeta^{-t,\star,\ast}_{\shift}(\overleftarrow{\bn''};T_{2})\zeta^{s,\sh}_{\shift}(\overleftarrow{\bm'},j+1;0)\zeta^{-t,\sh}_{\shift}(\bn',u-j;0),
\end{multlined}\\
f^{s,t}_{2}(\bl;T_{1},T_{2})&\coloneqq
\begin{multlined}[t]
\sum_{(\bn'',\bn',u,\bm^{(1)},\bm^{(2)},\bm'')=\bl}(-1)^{\wt(\bn'',\bn',\bm^{(1)})+\dep(\bn',\bm^{(2)})+u+1}\\
\cdot\zeta^{s,\star,\ast}_{\shift}(\bm'';T_{1})\zeta^{-t,\star,\ast}_{\shift}(\overleftarrow{\bn''};T_{2})\zeta^{s,\ast}_{\shift}(\overleftarrow{\bm^{(2)}};T_{1})\zeta^{-t,\ast}_{\shift}(\bn',u;T_{2})R(\bm^{(1)},1;T_{1}),
\end{multlined}\\
f^{s,t}_{3}(\bl;T_{1},T_{2})&\coloneqq
\begin{multlined}[t]
\sum_{(\bn'',\bn^{(2)},\bn^{(1)},u,\bm',\bm'')=\bl}(-1)^{\wt(\bn'',\bn^{(2)})+\dep(\bn^{(2)},\bm')}\\
\cdot\zeta^{s,\star,\ast}_{\shift}(\bm'';T_{1})\zeta^{-t,\star,\ast}_{\shift}(\overleftarrow{\bn''};T_{2})\zeta^{s,\ast}_{\shift}(\overleftarrow{\bm'},u;T_{1})\zeta^{-t,\ast}_{\shift}(\bn^{(2)};T_{2})R(\bn^{(1)},1;T_{2}),
\end{multlined}\\
f^{s,t}_{4}(\bl;T_{1},T_{2})&\coloneqq
\begin{multlined}[t]
\sum_{(\bn'',\bn^{(2)},\bn^{(1)},1,\bm^{(1)},\bm^{(2)},\bm'')=\bl}(-1)^{\wt(\bn'',\bn^{(2)},\bm^{(1)})+\dep(\bn^{(2)},\bm^{(2)})}\\
\cdot\zeta^{s,\star,\ast}_{\shift}(\bm'';T_{1})\zeta^{-t,\star,\ast}_{\shift}(\overleftarrow{\bn''};T_{2})\zeta^{s,\ast}_{\shift}(\overleftarrow{\bm^{(2)}};T_{1})\zeta^{-t,\ast}_{\shift}(\bn^{(2)};T_{2})R(\bm^{(1)},1;T_{1})R(\bn^{(1)},1;T_{2})
\end{multlined}
\end{align}
and omit $s,t,T_{1},T_{2}$ if there is no risk of confusion.
\end{Lem}
\begin{proof}
By applying the equalities
\[\zeta^{s,\star,\ast}_{\shift}(j+1,\bm;T_{1})=\sum_{(\bm',\bm'')=\bm}(-1)^{\dep(\bm')}\zeta^{s,\star,\ast}_{\shift}(\bm'';T_{1})\zeta^{s,\ast}_{\shift}(\overleftarrow{\bm'},j+1;T_{1})\]
and
\[\zeta^{-t,\star,\ast}_{\shift}(u-j,\overleftarrow{\bn};T_{2})=\sum_{(\bn'',\bn')=\bn}(-1)^{\dep(\bn')}\zeta^{-t,\star,\ast}_{\shift}(\overleftarrow{\bn''};T_{2})\zeta^{-t,\ast}_{\shift}(\bn',u-j;T_{2})\]
which is immediately obtained from Corollary \ref{antipode_relation}, we have
\begin{multline}\label{eq:f_after_antipode}
f(\bl)=\sum_{(\bn'',\bn',u,\bm',\bm'')=\bl}\sum_{j=0}^{u-1}(-1)^{\wt(\bn'',\bn')+\dep(\bn',\bm')+u-j}\\
\cdot\zeta^{s,\star,\ast}_{\shift}(\bm'';T_{1})\zeta^{-t,\star,\ast}_{\shift}(\overleftarrow{\bn''};T_{2})\zeta^{s,\ast}_{\shift}(\overleftarrow{\bm'},j+1;T_{1})\zeta^{-t,\ast}_{\shift}(\bn',u-j;T_{2}).
\end{multline}
We consider expanding the factor $\zeta^{s,\ast}_{\shift}(\overleftarrow{\bm'},j+1;T_{1})\zeta^{-t,\ast}_{\shift}(\bn',u-j;T_{2})$. Additionally, the equation \eqref{eq:explicit_reg_thm} establishes
\[\zeta^{s,\sh}_{\shift}(\overleftarrow{\bm'},j+1;0)=\zeta^{s,\ast}_{\shift}(\overleftarrow{\bm'},j+1;T_{1})+\delta_{j,0}\sum_{(\bm^{(1)},\bm^{(2)})=\bm'}\zeta^{s,\ast}_{\shift}(\overleftarrow{\bm^{(2)}};T_{1})R(\bm^{(1)},1;T_{1})\]
and
\[\zeta^{-t,\sh}_{\shift}(\bn',u-j;0)=\zeta^{-t,\ast}_{\shift}(\bn',u-j;T_{2})+\delta_{j,u-1}\sum_{(\bn^{(2)},\bn^{(1)})=\bn'}\zeta^{-t,\ast}_{\shift}(\bn^{(2)};T_{2})R(\bn^{(1)},1;T_{2}).\]
If we express these two equations as $A=B+C$ and $D=E+F$, respectively, what we should expand is $BE$ and this is executed as
\begin{align}
BE
&=(A-C)(D-F)\\
&=AD-CD-AF+CF\\
&=AD-C(E+F)-(B+C)F+CF\\
&=AD-CE-BF-CF,
\end{align}
which means
\begin{align}
&\zeta^{s,\ast}_{\shift}(\overleftarrow{\bm'},j+1;T_{1})\zeta^{-t,\ast}_{\shift}(\bn',u-j;T_{2})\\
&=\begin{multlined}[t]
\zeta^{s,\sh}_{\shift}(\overleftarrow{\bm'},j+1;0)\zeta^{-t,\sh}_{\shift}(\bn',u-j;0)\\
-\delta_{j,0}\zeta^{-t,\ast}_{\shift}(\bn',u-j;T_{2})\sum_{(\bm^{(1)},\bm^{(2)})=\bm'}\zeta^{s,\ast}_{\shift}(\overleftarrow{\bm^{(2)}};T_{1})R(\bm^{(1)},1;T_{1})\\
-\delta_{j,u-1}\zeta^{s,\ast}_{\shift}(\overleftarrow{\bm'},j+1;T_{1})\sum_{(\bn^{(2)},\bn^{(1)})=\bn'}\zeta^{-t,\ast}_{\shift}(\bn^{(2)};T_{2})R(\bn^{(1)},1;T_{2})\\
-\delta_{j,0}\delta_{j,u-1}\sum_{(\bm^{(1)},\bm^{(2)})=\bm'}\zeta^{s,\ast}_{\shift}(\overleftarrow{\bm^{(2)}};T_{1})R(\bm^{(1)},1;T_{1})\sum_{(\bn^{(2)},\bn^{(1)})=\bn'}\zeta^{-t,\ast}_{\shift}(\bn^{(2)};T_{2})R(\bn^{(1)},1;T_{2}).
\end{multlined}
\end{align}
Applying \eqref{eq:f_after_antipode}, we have
\begin{multline}
f(\bl)
=\sum_{(\bn'',\bn',u,\bm',\bm'')=\bl}\sum_{j=0}^{u-1}(-1)^{\wt(\bn'',\bn')+\dep(\bn',\bm')+u-j}\zeta^{s,\star,\ast}_{\shift}(\bm'';T_{1})\zeta^{-t,\star,\ast}_{\shift}(\overleftarrow{\bn''};T_{2})\\
\cdot\zeta^{s,\sh}_{\shift}(\overleftarrow{\bm'},j+1;0)\zeta^{-t,\sh}_{\shift}(\bn',u-j;0)\\
+\sum_{(\bn'',\bn',u,\bm',\bm'')=\bl}\sum_{j=0}^{u-1}(-1)^{\wt(\bn'',\bn')+\dep(\bn',\bm')+u-j+1}\zeta^{s,\star,\ast}_{\shift}(\bm'';T_{1})\zeta^{-t,\star,\ast}_{\shift}(\overleftarrow{\bn''};T_{2})\\
\cdot\delta_{j,0}\zeta^{-t,\ast}_{\shift}(\bn',u-j;T_{2})\sum_{(\bm^{(1)},\bm^{(2)})=\bm'}\zeta^{s,\ast}_{\shift}(\overleftarrow{\bm^{(2)}};T_{1})R(\bm^{(1)},1;T_{1})\\
+\sum_{(\bn'',\bn',u,\bm',\bm'')=\bl}\sum_{j=0}^{u-1}(-1)^{\wt(\bn'',\bn')+\dep(\bn',\bm')+u-j+1}\zeta^{s,\star,\ast}_{\shift}(\bm'';T_{1})\zeta^{-t,\star,\ast}_{\shift}(\overleftarrow{\bn''};T_{2})\\
\cdot\delta_{j,u-1}\zeta^{s,\ast}_{\shift}(\overleftarrow{\bm'},j+1;T_{1})\sum_{(\bn^{(2)},\bn^{(1)})=\bn'}\zeta^{-t,\ast}_{\shift}(\bn^{(2)};T_{2})R(\bn^{(1)},1;T_{2})\\
+\sum_{(\bn'',\bn',u,\bm',\bm'')=\bl}\sum_{j=0}^{u-1}(-1)^{\wt(\bn'',\bn')+\dep(\bn',\bm')+u-j+1}\zeta^{s,\star,\ast}_{\shift}(\bm'';T_{1})\zeta^{-t,\star,\ast}_{\shift}(\overleftarrow{\bn''};T_{2})\\
\cdot\delta_{j,0}\delta_{j,u-1}\sum_{(\bm^{(1)},\bm^{(2)})=\bm'}\zeta^{s,\ast}_{\shift}(\overleftarrow{\bm^{(2)}};T_{1})R(\bm^{(1)},1;T_{1})\sum_{(\bn^{(2)},\bn^{(1)})=\bn'}\zeta^{-t,\ast}_{\shift}(\bn^{(2)};T_{2})R(\bn^{(1)},1;T_{2}).
\end{multline}
The first term is $f_{1}(\bl)$ by definition. The $i$-th term on the right-hand side ($i\in\{2,3,4\}$) also coincides with $f_{i}(\bl)$ owing to Kronecker's deltas. We remark that some visual gaps of sign between the definition of $f_{i}(\bl)$ and the $i$-th term are canceled because $R(\bk;T)$ vanishes when $\bk$ has an entry that is not $1$; for example, in the definition of $f_{2}(\bl)$, the power of sign is
\[\wt(\bn'',\bn',\bm^{(1)})+\dep(\bn',\bm^{(2)})+u+1=\wt(\bn'',\bn')+\dep(\bn',\bm')+u-j+1\quad ((\bm^{(1)},\bm^{(2)})=\bm)\]
when $j=0$ and any entry of $\bm^{(1)}$ is $1$.
\end{proof}
We now compute $f_{i}(\bl)$ for $i\in\{1,2,3,4\}$. First, we address $i=2,3,4$ using Theorem \ref{antipode_relation}.
\begin{Lem}\label{explicit_f234}
For an index $\bl$ with $\wt(\bl)>\dep(\bl)$, we have the followings.
\begin{enumerate}[\rm(1)]
\item\label{f3} $f^{s,t}_{2}(\bl;T_{1},T_{2})=(-1)^{\wt(\bl)+1}f^{-t,-s}_{3}(\overleftarrow{\bl};T_{2},T_{1})$.
\item\label{f4} $f_{4}(\bl)=0$.
\item\label{f2}
\[f_{2}(\bl)=(-1)^{\wt(\bl)+1}\sum_{(\bp,\bm^{(1)})=\bl}\zeta^{-t,\star,\ast}_{\shift}(\overleftarrow{\bp};T_{2})R(\bm^{(1)},1;T_{1}).\]
\end{enumerate}
\end{Lem}
\begin{proof}
(\ref{f3}) is obvious from the definition. By substituting $(\bn'',\bn^{(2)})=\bn'''$ and $(\bm^{(2)},\bm'')=\bm'''$, it holds that
\begin{multline}
f_{4}(\bl)
=\sum_{(\bn''',\bn^{(1)},1,\bm^{(1)},\bm''')=\bl}(-1)^{\wt(\bn''',\bm^{(1)})}R(\bn^{(1)},1;T_{2})R(\bm^{(1)},1;T_{1})\\
\cdot\left[\sum_{(\bn'',\bn^{(2)})=\bn'''}(-1)^{\dep(\bn^{(2)})}\zeta^{-t,\star,\ast}_{\shift}(\overleftarrow{\bn''};T_{2})\zeta^{-t,\ast}_{\shift}(\bn^{(2)};T_{2})\right]\\
\cdot\left[\sum_{(\bm^{(2)},\bm'')=\bm'''}(-1)^{\dep(\bm^{(2)})}\zeta^{s,\star,\ast}_{\shift}(\bm'';T_{1})\zeta^{s,\ast}_{\shift}(\overleftarrow{\bm^{(2)}};T_{1})\right].
\end{multline}
Applying Corollary \ref{antipode_relation} to the expression in brackets yields (\ref{f4}). We can prove (\ref{f2}) with the same idea: compute $f_{2}(\bl)$ as
\begin{multline}
f_{2}(\bl)
=\sum_{(\bp,u,\bm^{(1)},\bp')=\bl}(-1)^{\wt(\bp,\bm^{(1)})+u+1}R(\bm^{(1)},1;T_{1})\left[\sum_{(\bn'',\bn')=\bp}(-1)^{\dep(\bn')}\zeta^{-t,\star,\ast}_{\shift}(\overleftarrow{\bn''};T_{2})\zeta^{-t,\ast}_{\shift}(\bn',u;T_{2})\right]\\
\cdot\left[\sum_{(\bm^{(2)},\bm'')=\bp'}(-1)^{\dep(\bm^{(2)})}\zeta^{s,\star,\ast}_{\shift}(\bm'';T_{1})\zeta^{s,\ast}_{\shift}(\overleftarrow{\bm^{(2)}};T_{1})\right]
\end{multline}
and use Corollary \ref{antipode_relation}.
\end{proof}
For an index $\bl$, let
\[g(\bl)\coloneqq\sum_{(\bn',u,\bm')=\bl}\sum_{j=0}^{u-1}(-1)^{\wt(\bn')+u-j}\zeta^{-t,\sh}_{\shift}(\bn',u-j;0)\zeta^{s,\sh}_{\shift}(\overleftarrow{\bm'},j+1;0).\]
We remark that $g(\emp)=0$. Then the definition of $f_{1}(\bl)$ can be written as
\begin{equation}\label{eq:connection_f1_g}
f_{1}(\bl)=\sum_{(\bn'',\bl',\bm'')=\bl}(-1)^{\wt(\bn'')+\dep(\bl')-1}\zeta^{-t,\star,\ast}_{\shift}(\overleftarrow{\bn''};T_{2})\zeta^{s,\star,\ast}_{\shift}(\bm'';T_{1})g(\bl').
\end{equation}
\begin{Lem}\label{shuffle_lemma}
For a word $w=e_{a_{1}}\cdots e_{a_{N}}$ ($a_{1},\ldots,a_{N}\in\{0,1\}$), we have
\begin{align}
&\sum_{v=0}^{N}(-1)^{v}\frac{1}{1+e_{0}t}e_{1}e_{a_{1}}\cdots e_{a_{v}}\sh\frac{1}{1-e_{0}s}e_{1}e_{a_{N}}\cdots e_{a_{v+1}}\\
&=\frac{1}{1+e_{0}t}\sh\left(\frac{1}{1-e_{0}s}e_{1}\overleftarrow{w}e_{1}\frac{1}{1-e_{0}t}\right)+(-1)^{N}\frac{1}{1-e_{0}s}\sh\left(\frac{1}{1+e_{0}t}e_{1}we_{1}\frac{1}{1+e_{0}s}\right).
\end{align}
\end{Lem}
\begin{proof}
The definition of the shuffle product ensures that the left-hand side is telescopic. Thus, we obtain
\begin{align}
&\sum_{v=0}^{N}(-1)^{v}\frac{1}{1+e_{0}t}e_{1}e_{a_{1}}\cdots e_{a_{v}}\sh\frac{1}{1-e_{0}s}e_{1}e_{a_{N}}\cdots e_{a_{v+1}}\\
&=\left(\frac{1}{1+e_{0}t}\sh\frac{1}{1-e_{0}s}e_{1}\overleftarrow{w}\right)e_{1}+(-1)^{N}\left(\frac{1}{1+e_{0}t}e_{1}w\sh\frac{1}{1-e_{0}s}\right)e_{1}.
\end{align}
Thus, it suffices to prove that
\begin{equation}\label{eq:shuffle_lemma_half}
\left(\frac{1}{1+e_{0}t}e_{1}w\sh\frac{1}{1-e_{0}s}\right)e_{1}=\frac{1}{1-e_{0}s}\sh\left(\frac{1}{1+e_{0}t}e_{1}we_{1}\frac{1}{1+e_{0}s}\right),
\end{equation}
because the remainder of the claim follows by substituting $s\mapsto -t$, $t\mapsto -s$ and $w\mapsto\overleftarrow{w}$.
By the combinatorial description of the shuffle product, the right-hand side of \eqref{eq:shuffle_lemma_half} is equal to
\[
\left(\frac{1}{1-e_{0}s}\sh\frac{1}{1+e_{0}t}e_{1}w\right)e_{1}\left(\frac{1}{1-e_{0}s}\sh\frac{1}{1+e_{0}s}\right).
\]
Since $\frac{1}{1-e_{0}s}\sh\frac{1}{1+e_{0}s}=1$, this is equal to the left-hand side of \eqref{eq:shuffle_lemma_half}.
\end{proof}
\begin{Lem}\label{explicit_g}
For an index $\bl$, we have
\[g(\bl)=-\sum_{j=0}^{\infty}\zeta^{s,\sh}_{\shift}(\overleftarrow{\bl},j+1)t^{j}+(-1)^{\wt(\bl)}\sum_{j=0}^{\infty}\zeta^{-t,\sh}_{\shift}(\bl,j+1)(-s)^{j}.\]
\end{Lem}
\begin{proof}
The case $\bl=\emp$ follows from the definition. Assume that $\bl \neq \emp$.
Define $a_{1},\ldots,a_{N}\in\{0,1\}$ by $\omega(\bl)=e_{a_{1}}\cdots e_{a_{N}}$ and let $Z^{\sh}$ and  $\sh$ be defined in $\fH\jump{s,t}$ coefficientwise. By \eqref{eq:shifted_mzv_coefficient}, we have
\begin{align}
g(\bl)
&=\sum_{v=0}^{N}(-1)^{v+\dep(\bl)}Z^{\sh}\left(\frac{1}{1+e_{0}t}e_{1}e_{a_{1}}\cdots e_{a_{v}}\right)Z^{\sh}\left(\frac{1}{1-e_{0}s}e_{1}e_{a_{N}}\cdots e_{a_{v+1}}\right)\\
&=(-1)^{\dep(\bl)}Z^{\sh}\left(\sum_{v=0}^{N}(-1)^{v}\frac{1}{1+e_{0}t}e_{1}e_{a_{1}}\cdots e_{a_{v}}\sh\frac{1}{1-e_{0}s}e_{1}e_{a_{N}}\cdots e_{a_{v+1}}\right).
\end{align}
Using Lemma \ref{shuffle_lemma} and the fact $Z^{\sh}(e_{0}^{n})=\delta_{n,0}$ ($n\ge 0$), we obtain the result.
\end{proof}
\begin{Lem}\label{explicit_f}
For an index $\bl$ with $\wt(\bl)>\dep(\bl)$, we have
\begin{multline}
f(\bl)=\sum_{(\ba,\bb)=\bl}\sum_{j=0}^{\infty}(-1)^{\wt(\ba)}\zeta^{-t,\star,\ast}_{\shift}(\overleftarrow{\ba};T_{2})\zeta^{s,\star,\ast}_{\shift}(j+1,\bb;T_{1})t^{j}\\
-\sum_{(\ba,\bb)=\bl}\sum_{j=0}^{\infty}(-1)^{\wt(\ba)}\zeta^{s,\star,\ast}_{\shift}(\bb;T_{2})\zeta^{-t,\star,\ast}_{\shift}(j+1,\overleftarrow{\ba};T_{1})(-s)^{j}.
\end{multline}
\end{Lem}
\begin{proof}
Put
\[f^{s,t}_{1,1}(\bl;T_{1},T_{2})\coloneqq\sum_{(\bn'',\bl',\bm'')=\bl}\sum_{j=0}^{\infty}(-1)^{\wt(\bn'')+\dep(\bl')}\zeta^{-t,\star,\ast}_{\shift}(\overleftarrow{\bn''};T_{2})\zeta^{s,\star,\ast}_{\shift}(\bm'';T_{1})\zeta^{s,\sh}_{\shift}(\overleftarrow{\bl'},j+1;0)t^{j}\]
and
\[f^{s,t}_{1,2}(\bl;T_{1},T_{2})\coloneqq\sum_{(\bn'',\bl',\bm'')=\bl}\sum_{j=0}^{\infty}(-1)^{\wt(\bm'')+\dep(\bl')}\zeta^{-t,\star,\ast}_{\shift}(\overleftarrow{\bn''};T_{2})\zeta^{s,\star,\ast}_{\shift}(\bm'';T_{1})\zeta^{-t,\sh}_{\shift}(\bl',j+1;0)(-s)^{j}\]
for a suitable index $\bl$. Then we have $f^{s,t}_{1,1}(\bl;T_{1},T_{2})=f^{-t,-s}_{1,2}(\overleftarrow{\bl};T_{2},T_{1})$ by definition and $f_{1}(\bl)=f^{s,t}_{1,1}(\bl;T_{1},T_{2})+(-1)^{\wt(\bl)+1}f^{s,t}_{1,2}(\bl;T_{1},T_{2})$ from Lemma \ref{explicit_g}. Moreover, \eqref{eq:explicit_reg_thm} shows
\[\zeta^{s,\sh}_{\shift}(\overleftarrow{\bl'},j+1;0)=\zeta^{s,\ast}_{\shift}(\overleftarrow{\bl'},j+1;T_{1})+\delta_{j,0}\sum_{(\ba,\bb)=\bl'}\zeta^{s,\ast}_{\shift}(\overleftarrow{\bb};T_{1})R(\ba,1;T_{1}).\]
Using Corollary \ref{antipode_relation} and this identity, we have
\begin{align}
&f_{1,1}^{s,t}(\bl;T_{1},T_{2})\\
&=\begin{multlined}[t]
\sum_{(\bn'',\bl',\bm'')=\bl}\sum_{j=0}^{\infty}(-1)^{\wt(\bn'')+\dep(\bl')}\zeta^{-t,\star,\ast}_{\shift}(\overleftarrow{\bn''};T_{2})\zeta^{s,\star,\ast}_{\shift}(\bm'';T_{1})\zeta^{s,\ast}_{\shift}(\overleftarrow{\bl'},j+1;T_{1})t^{j}\\
+\sum_{(\bn'',\ba,\bb,\bm'')=\bl}(-1)^{\wt(\bn'')+\dep(\ba,\bb)}\zeta^{-t,\star,\ast}_{\shift}(\overleftarrow{\bn''};T_{2})\zeta^{s,\star,\ast}_{\shift}(\bm'';T_{1})\zeta^{s,\ast}_{\shift}(\overleftarrow{\bb};T_{1})R(\ba,1;T_{1})
\end{multlined}\\
&=\begin{multlined}[t]
\sum_{(\bn'',\bp)=\bl}\sum_{j=0}^{\infty}(-1)^{\wt(\bn'')}\zeta^{-t,\star,\ast}_{\shift}(\overleftarrow{\bn''};T_{2})\left[\sum_{(\bl',\bm'')=\bp}(-1)^{\dep(\bl')}\zeta^{s,\star,\ast}_{\shift}(\bm'';T_{1})\zeta^{s,\ast}_{\shift}(\overleftarrow{\bl'},j+1;T_{1})\right]t^{j}\\
+\sum_{(\bn'',\ba,\bp)=\bl}(-1)^{\wt(\bn'')+\dep(\ba)}\zeta^{-t,\star,\ast}_{\shift}(\overleftarrow{\bn''};T_{2})R(\ba,1;T_{1})\left[\sum_{(\bb,\bm'')=\bp}(-1)^{\dep(\bb)}\zeta^{s,\star,\ast}_{\shift}(\bm'';T_{1})\zeta^{s,\ast}_{\shift}(\overleftarrow{\bb};T_{1})\right]
\end{multlined}\\
&=\sum_{(\bn'',\bp)=\bl}\sum_{j=0}^{\infty}(-1)^{\wt(\bn'')}\zeta^{-t,\star,\ast}_{\shift}(\overleftarrow{\bn''};T_{2})\zeta^{s,\star,\ast}_{\shift}(j+1,\bp;T_{1})t^{j}+(-1)^{\wt(\bl)}\sum_{(\bn'',\ba)=\bl}\zeta^{-t,\star,\ast}_{\shift}(\overleftarrow{\bn''};T_{2})R(\ba,1;T_{1}).
\end{align}
Here we used $\dep(\ba)=\wt(\ba)$ in the last term unless it vanishes. Combining this computation and Lemma \ref{explicit_f234} (\ref{f2}) yields
\[f^{s,t}_{1,1}(\bl;T_{1},T_{2})+f^{s,t}_{2}(\bl;T_{1},T_{2})=\sum_{(\ba,\bb)=\bl}\sum_{j=0}^{\infty}(-1)^{\wt(\ba)}\zeta^{-t,\star,\ast}_{\shift}(\overleftarrow{\ba};T_{2})\zeta^{s,\star,\ast}_{\shift}(j+1,\bb;T_{1})t^{j}.\]
Consequently, we get
\begin{align}
f(\bl)
&=f^{s,t}_{1}(\bl;T_{1},T_{2})+f^{s,t}_{2}(\bl;T_{1},T_{2})+f^{s,t}_{3}(\bl;T_{1},T_{2})+f^{s,t}_{4}(\bl;T_{1},T_{2})\\
&=f^{s,t}_{1,1}(\bl;T_{1},T_{2})+f^{s,t}_{2}(\bl;T_{1},T_{2})+(-1)^{\wt(\bl)+1}(f^{-t,-s}_{1,1}(\overleftarrow{\bl};T_{2},T_{1})+f^{-t,-s}_{2}(\overleftarrow{\bl};T_{2},T_{1}))\\
&=\begin{multlined}[t]
\sum_{(\ba,\bb)=\bl}\sum_{j=0}^{\infty}(-1)^{\wt(\ba)}\zeta^{-t,\star,\ast}_{\shift}(\overleftarrow{\ba};T_{2})\zeta^{s,\star,\ast}_{\shift}(j+1,\bb;T_{1})t^{j}\\
-\sum_{(\ba,\bb)=\bl}\sum_{j=0}^{\infty}(-1)^{\wt(\ba)}\zeta^{s,\star,\ast}_{\shift}(\bb;T_{1})\zeta^{-t,\star,\ast}_{\shift}(j+1,\overleftarrow{\ba};T_{2})(-s)^{j}
\end{multlined}
\end{align}
by Lemma \ref{explicit_f234} (\ref{f3}), (\ref{f4}) and Lemma \ref{f_decomposition}.
\end{proof}
Lemma \ref{explicit_f} implies that
\begin{align}
F^{\ast}(\bk)
&=\begin{multlined}[t]
\sum_{\bl\in\alpha(\bk)}\Biggl(\sum_{(\ba,\bb)=\bl}\sum_{j=0}^{\infty}(-1)^{\wt(\ba)}\zeta^{-t,\star,\ast}_{\shift}(\overleftarrow{\ba};T_{2})\zeta^{s,\star,\ast}_{\shift}(j+1,\bb;T_{1})t^{j}\Biggr.\\
\Biggl.-\sum_{(\ba,\bb)=\bl}\sum_{j=0}^{\infty}(-1)^{\wt(\ba)}\zeta^{s,\star,\ast}_{\shift}(\bb;T_{1})\zeta^{-t,\star,\ast}_{\shift}(j+1,\overleftarrow{\ba};T_{2})(-s)^{j}\Biggr)
\end{multlined}\\
&=\begin{multlined}[t]
\sum_{\bl\in\alpha(\bk)}\Biggl(\sum_{(\bb,\ba)=\bl}\sum_{j=0}^{\infty}(-1)^{\wt(\ba)}\zeta^{-t,\star,\ast}_{\shift}(\overleftarrow{\ba};T_{2})\zeta^{s,\star,\ast}_{\shift}(j+1,\bb;T_{1})t^{j}\Biggr.\\
\Biggl.-\sum_{(\bb,\ba)=\bl}\sum_{j=0}^{\infty}(-1)^{\wt(\ba)}\zeta^{s,\star,\ast}_{\shift}(\bb;T_{1})\zeta^{-t,\star,\ast}_{\shift}(j+1,\overleftarrow{\ba};T_{2})(-s)^{j}\Biggr)
\end{multlined}\\
&=\begin{multlined}[t]
\sum_{\bl\in\alpha(\bk)}\sum_{j=0}^{\infty}\left(\zeta^{s,t,\star,\ast}_{\hcS}(j+1,\bl;T_{1},T_{2})-(-1)^{k+j+1}\zeta^{-t,\star,\ast}_{\shift}(\overleftarrow{\bl},j+1;T_{2})\right)t^{j}\\
+\sum_{\bl\in\alpha(\bk)}\sum_{j=0}^{\infty}\left(\zeta^{s,t,\star,\ast}_{\hcS}(\bl,j+1;T_{1},T_{2})-\zeta^{s,\star,\ast}_{\shift}(\bl,j+1;T_{1})\right)s^{j}
\end{multlined}
\end{align}
holds. Hence we obtain the case where $\bullet=\ast$ in \eqref{eq:definition_of_F} and Theorem \ref{general_stadic_csf}.
\subsection{Proof of the $(s,t)$-adic $\tau$-interpolated cyclic sum formula}\label{subset:cyclic_sum_formula_tau}
We can also prove the cyclic sum formula for $(s,t)$-adic symmetric multiple zeta (non-star) values:
\begin{Thm}\label{general_stadic_nonstar_csf}
For an index $\bk=(k_{1},\ldots,k_{r})$ whose weight $k$ is greater than $r$, we have
\begin{align}
&\sum_{i=1}^{r}\sum_{j=0}^{k_i-1}\zeta^{s,t,\bullet}_{\hcS}(j+1,\bk^{[i]},\bk_{[i-1]},k_{i}-j;T_{1},T_{2})\\
&=\begin{multlined}[t]
\sum_{i=1}^{r}\sum_{j=0}^{\infty}\left(\zeta^{s,t,\bullet}_{\hcS}(j+1,\bk^{[i]},\bk_{[i]};T_{1},T_{2})+\zeta^{s,t,\bullet}_{\hcS}((j+1)\uplus\bk^{[i]},\bk_{[i]};T_{1},T_{2})\right)t^{j}\\
+\sum_{i=1}^{r}\sum_{j=0}^{\infty}\left(\zeta^{s,t,\bullet}_{\hcS}(\bk^{[i]},\bk_{[i]},j+1;T_{1},T_{2})+\zeta^{s,t,\bullet}_{\hcS}(\bk^{[i]},\bk_{[i]}\uplus (j+1);T_{1},T_{2})\right)s^{j}.
\end{multlined}
\end{align}
Here, for non-empty indices $\bk=(k_{1},\ldots,k_{r})$ and $\bl=(l_{1},\ldots,l_{r'})$, we denote by $\bk\uplus\bl$ an index
\[(k_{1},\ldots,k_{r-1},k_{r}+l_{1},l_{2},\ldots,l_{r'}).\]
\end{Thm}
From this theorem, one can recover Hirose--Murahara--Ono's $t$-adic cyclic sum formula \cite[Theorem 6.1]{hmo21} by putting $s=0$. For $\tau\in\bbQ$,
define the $\tau$-interpolated $(s,t)$-adic
symmetric multiple zeta values by
\[
\zeta_{\hcS}^{s,t,\tau,\bullet}(\bk;T_{1},T_{2})=\sum_{\bl\preceq\bk}\zeta_{\hcS}^{s,t,\bullet}(\bl;T_{1},T_{2})\tau^{\dep(\bk)-\dep(\bl)}
\]
(see \cite{yamamoto13} for the original $\tau$-interpolated multiple zeta values).
Note that $\zeta_{\hcS}^{s,t,\tau,\bullet}(\bk;T_{1},T_{2})$ coincides
with $\zeta_{\hcS}^{s,t,\bullet}(\bk;T_{1},T_{2})$ (resp. $\zeta_{\hcS}^{s,t,\star,\bullet}(\bk;T_{1},T_{2})$)
when $\tau=0$ (resp. $\tau=1$).
Theorem \ref{general_stadic_nonstar_csf} can be derived from the following $\tau$-interpolated $(s,t)$-adic cyclic sum formula:
\begin{Thm}\label{general_stadic_tau_csf}
For an index $\bk=(k_{1},\ldots,k_{r})$ whose weight $k$ is greater than $r$, we have
\begin{align}
&\sum_{i=1}^{r}\sum_{j=0}^{k_i-1}\zeta^{s,t,\tau,\bullet}_{\hcS}(j+1,\bk^{[i]},\bk_{[i-1]},k_{i}-j;T_{1},T_{2}) - k\tau^r\zeta^{s,t,\star,\bullet}_{\hcS}(k+1;T_{1},T_{2}).  \\
&=\begin{multlined}[t]
\sum_{i=1}^{r}\sum_{j=0}^{\infty}\left(\zeta^{s,t,\tau,\bullet}_{\hcS}(j+1,\bk^{[i]},\bk_{[i]};T_{1},T_{2})+(1-\tau)\zeta^{s,t,\tau,\bullet}_{\hcS}((j+1)\uplus\bk^{[i]},\bk_{[i]};T_{1},T_{2})\right)t^{j}\\
+\sum_{i=1}^{r}\sum_{j=0}^{\infty}\left(\zeta^{s,t,\tau,\bullet}_{\hcS}(\bk^{[i]},\bk_{[i]},j+1;T_{1},T_{2})+(1-\tau)\zeta^{s,t,\tau,\bullet}_{\hcS}(\bk^{[i]},\bk_{[i]}\uplus (j+1);T_{1},T_{2})\right)s^{j}.
\end{multlined}
\end{align}
\end{Thm}

In the rest part of this section, we prove Theorem \ref{general_stadic_tau_csf}.
We define $\mathbb{Q}$-algebra endomorphisms $S^{\tau}$ and $A^{\tau}$
of $\fH$ by
\[
S^{\tau}(e_{1})=e_{1}+\tau e_{0},\ A^{\tau}(e_{1})=\tau e_{0},\ S^{\tau}(e_{0})=A^{\tau}(e_{0})=e_{0},
\]
and $\mathbb{Q}$-linear endomorphisms $C$ and $H$ of $\fH$ by
\[
C(e_{a_{1}}\cdots e_{a_{r}})=\sum_{j=1}^{r}e_{a_{j+1}}\cdots e_{a_{r}}e_{a_{1}}\cdots e_{a_{j}},\qquad(a_{1},\dots,a_{r}\in\{0,1\}),
\]
\[
H(1)=H(e_{0}w)=0,\quad H(e_{1}w)=w\qquad(w\in\fH).
\]
They are naturally extended to the endomorphisms of $\fH\jump{s,t}$.
Furthermore, define $\mathbb{Q}\jump{s,t}$-linear continuous maps
$F^{s,t,\tau,\bullet}:\fH\jump{s,t}\to\bbR[T_{1},T_{2}]\jump{s,t}$
and $L^{s,t,\tau}:\fH\jump{s,t}\to\fH\jump{s,t}$ by
\[
F^{s,t,\tau,\bullet}(e_{0}^{k_{0}-1}e_{1}e_{0}^{k_{1}-1}\cdots e_{1}e_{0}^{k_{r}-1})=\zeta_{\widehat{\mathcal{S}}}^{s,t,\tau,\bullet}(k_{0},\dots,k_{r}),
\]
\[
L^{s,t,\tau}(w)=w-A^{\tau}(w)-\sum_{j=0}^{\infty}e_{0}^{j}(e_{1}+(1-\tau)e_{0})H(w)t^{j}-\sum_{j=0}^{\infty}H(w)(e_{1}+(1-\tau)e_{0})e_{0}^{j}s^{j}.
\]
Then, for $k_{1}+\cdots+k_{r}>r$, we have
\[
F^{s,t,1,\bullet}\circ L^{s,t,1}\circ C(u)=0\qquad(u\text{ is any rotation of }e_{1}e_{0}^{k_{1}-1}\cdots e_{1}e_{0}^{k_{r}-1})
\]
by Theorem \ref{general_stadic_csf}. Furthermore, $F^{s,t,1,\bullet}\circ L^{s,t,1}\circ C(e_{0}^{r})=F^{s,t,1,\bullet}(0)=0$
for $r\geq0$. Thus
\[
F^{s,t,1,\bullet}\circ L^{s,t,1}\circ C(w)=0\qquad(w\in\fH'\coloneqq\fH e_{0}\fH\jump{s,t})\label{eq:F1L1c_vanish}
\]
where $\fH'\subset\fH$ is the subspace generated by words which contain
$e_{0}$. By definition, we have
\[
F^{s,t,\tau,\bullet}=F^{s,t,0,\bullet}\circ S^{\tau},\quad S^{\tau}\circ L^{s,t,\tau}=L^{s,t,0}\circ S^{\tau},\quad S^{\tau}\circ C=C\circ S^{\tau},
\]
which imply
\[
F^{s,t,\tau,\bullet}\circ L^{s,t,\tau}\circ C=F^{s,t,0,\bullet}\circ L^{s,t,0}\circ C\circ S^{\tau}.
\]
Thus, for all $w\in\fH'$, we have
\[
F^{s,t,\tau,\bullet}\circ L^{s,t,\tau}\circ C(w)=F^{s,t,0,\bullet}\circ L^{s,t,0}\circ C\circ S^{\tau}(w)=F^{s,t,1,\bullet}\circ L^{s,t,1}\circ C(S^{\tau-1}(w))=0,
\]
which implies Theorem \ref{general_stadic_tau_csf} by putting $w=e_{1}e_{0}^{k_{1}-1}\cdots e_{1}e_{0}^{k_{r}-1}$.

\section{Finite analogue}
We discuss a finite counterpart of $(s,t)$-adic symmetric multiple zeta values.
Let us give a review of the refined Kaneko--Zagier conjecture. For a positive integer $n$, define a $\bbQ$-algebra $\cA_{n}$ by
\[\cA_{n}\coloneqq\left(\prod_{p}\bbZ/p^{n}\bbZ\right)\biggm/\left(\bigoplus_{p}\bbZ/p^{n}\bbZ\right),\]
where $p$ runs all prime numbers.
Then we can define the projective limit $\hcA\coloneqq\varprojlim_{n}\cA_{n}$ by natural projections $\cA_{n+1}\to\cA_{n}$.
Let $\ilp\coloneqq((p\mod p^{n})_{p})_{n}\in\hcA$.
As an element of $\hcA$, put
\[\zeta_{\hcA}(\bk)\coloneqq\left(\left(\sum_{0<n_{1}<\cdots<n_{r}<p}\frac{1}{n_{1}^{k_{1}}\cdots n_{r}^{k_{r}}}\mod p^{n}\right)_{p}\right)_{n}\]
for an index $\bk=(k_{1},\ldots,k_{r})$, and call it the \emph{$\ilp$-adic finite multiple zeta value}.
Let $\cZ_{\hcA}$ be the $\bbQ$-topological algebra (equipped with the $\ilp$-adic topology) generated by all $\ilp$-adic finite MZVs (including $\zeta_{\hcA}(\emp)\coloneqq 1$) and $\ilp$.
\begin{Conj}[{\cite[Conjecture 4.3]{osy21}}, {\cite[Conjecture 5.3.2]{jarossay19}}, {\cite[Conjecture 2.3]{rosen19}}]\label{refined_kz_conj}
There is a topological $\bbQ$-algebra isomorphism $\hper\colon(\cZ/\zeta(2)\cZ)\jump{t}\to\cZ_{\hcA}$ sending $\zeta_{\hcS}(\bk)$ to $\zeta_{\hcA}(\bk)$ and $t$ to $\ilp$.
\end{Conj}
For an integer $a$, we define an element $\zeta_{\hcA}(\bk;a)\in\hcA$ by
\begin{equation}\label{eq:definition_fmzv_with_a}
\zeta_{\hcA}(\bk;a)\coloneqq\left(\left(\sum_{ap<n_{1}<\cdots<n_{r}<(a+1)p}\frac{1}{n_{1}^{k_{1}}\cdots n_{r}^{k_{r}}}\mod p^{n}\right)_{p}\right)_{n}.
\end{equation}
Obviously $\zeta_{\hcA}(\bk;0)=\zeta_{\hcA}(\bk)$ holds.
We suggest these values as a candidate of the finite analogue of $(s,t)$-adic SMZVs:
\begin{Thm}\label{st_kz_conj}
Let $\bk$ be an index.
Then $\zeta^{s,t}_{\hcS}(\bk)$ is the unique element of $(\cZ/\zeta(2)\cZ)\jump{s,t}$ satisfying
\begin{equation}\label{eq:st_kz_conj}
\hper(\zeta^{at,(a+1)t}_{\hcS}(\bk))= \zeta_{\hcA}(\bk;a)
\end{equation}
for any integer $a$ under Conjecture \ref{refined_kz_conj}.
\end{Thm}
\begin{proof}
Fix an index $\bk=(k_{1},\ldots,k_{r})$.
Since the $p$-component of $\zeta_{\hcA}(\bk;a)$ has the series expression
\begin{align}
\sum_{ap<n_{1}<\cdots<n_{r}<(a+1)p}\frac{1}{n_{1}^{k_{1}}\cdots n_{r}^{k_{r}}}
&=\sum_{0<n_{1}<\cdots<n_{r}<p}\frac{1}{(n_{1}+ap)^{k_{1}}\cdots (n_{r}+ap)^{k_{r}}}\\
&=\sum_{\bm=(m_{1},\ldots,m_{r})\in\bbZ^{r}_{\ge 0}}b\left({\bk\atop\bm}\right)\left(\sum_{0<n_{1}<\cdots<n_{r}<p}\frac{1}{n_{1}^{k_{1}+m_{1}}\cdots n_{r}^{k_{r}+m_{r}}}\right)(-ap)^{\wt(\bm)}
\end{align}
in $\bbZ_{p}$, we compute $\zeta_{\hcA}(\bk;a)$ by Conjecture \ref{refined_kz_conj} as
\begin{align}
\zeta_{\hcA}(\bk;a)
&=\sum_{\bm\in\bbZ^{r}_{\ge 0}}b\left({\bk\atop\bm}\right)\zeta_{\hcA}(\bk\oplus\bm)(-a\ilp)^{\wt(\bm)}\\
&=\sum_{\bm\in\bbZ^{r}_{\ge 0}}b\left({\bk\atop\bm}\right)\hper\left(\sum_{i=0}^{r}(-1)^{\wt(\bk^{[i]}\oplus\bm^{[i]})}\zeta^{\bullet}(\bk_{[i]}\oplus\bm_{[i]})\zeta^{-t,\bullet}_{\shift}(\overleftarrow{\bk^{[i]}\oplus\bm^{[i]}})\right)(-a\ilp)^{\wt(\bm)}\\
&=\begin{multlined}[t]
\sum_{i=0}^{r}(-1)^{\wt(\bk^{[i]})}\left(\sum_{\bm\in\bbZ^{i}_{\ge 0}}b\left({\bk_{[i]}\atop\bm}\right)\hper(\zeta^{\bullet}(\bk_{[i]}\oplus\bm))(-a\ilp)^{\wt(\bm)}\right)\\
\cdot\left(\sum_{\bm,\bn\in\bbZ^{r-i}_{\ge 0}}b\left({\bk^{[i]}\atop\bm}\right)b\left({\bk^{[i]}\oplus\bm\atop\bn}\right)\hper(\zeta^{\bullet}(\overleftarrow{\bk^{[i]}\oplus\bm\oplus\bn}))(a\ilp)^{\wt(\bm)}\ilp^{\wt(\bn)}\right).
\end{multlined}
\end{align}
On the other hand, we have
\begin{align}
\sum_{\bm\in\bbZ^{i}_{\ge 0}}b\left({\bk_{[i]}\atop\bm}\right)\hper(\zeta^{\bullet}(\bk_{[i]}\oplus\bm))(-a\ilp)^{\wt(\bm)}&=\hper(\zeta^{at,\bullet}_{\shift}(\bk_{[i]})),\\
\sum_{\bm,\bn\in\bbZ^{r-i}_{\ge 0}}b\left({\bk^{[i]}\atop\bm}\right)b\left({\bk^{[i]}\oplus\bm\atop\bn}\right)\hper(\zeta^{\bullet}(\overleftarrow{\bk^{[i]}\oplus\bm\oplus\bn}))(a\ilp)^{\wt(\bm)}\ilp^{\wt(\bn)}&=\hper(\zeta^{-(a+1)t,\bullet}_{\shift}(\overleftarrow{\bk^{[i]}})).
\end{align}
These equations establishes \eqref{eq:st_kz_conj}.
We prove the uniqueness: let $m$ and $n$ be non-negative integers, and $v(m,n)$ be an element of $\cZ/\zeta(2)\cZ$.
It is enough to show that if
\begin{equation}\label{eq:hypothesis_uniqueness}
\hper\left(\sum_{m,n=0}^{\infty}v(m,n)(at)^{m}((a+1)t)^{n}\right)=0
\end{equation}
for any $a\in\bbZ_{\ge 0}$, then $v(m,n)=0$ for any $m,n\in\bbZ_{\ge 0}$.
Assume that \eqref{eq:hypothesis_uniqueness} holds.
Then the injectivity of $\hper$ (Conjecture \ref{st_kz_conj}) shows
\[\sum_{i=0}^{N}v(i,N-i)a^{i}(a+1)^{N-i}=0.\]
In particular, we obtain
\[\sum_{i=0}^{N}v(i,N-i)\left(\frac{a}{a+1}\right)^{i}=0\]
for $a\in\bbZ_{\ge 0}$.
Thus the regularity of the Vandermonde matrix yields $v(i,N-i)=0$ for any $N\ge 0$ and $0\le i\le N$.
\end{proof}
\begin{Rem}
    Putting $a=0$ in Theorem \ref{st_kz_conj}, we have $\hper(\zeta_{\hcS}^{0,t}(\bk))=\zeta_{\hcA}(\bk)$. There are some known formulas for $\ilp$-adic FMZVs: the $\ilp$-adic harmonic relation, $\ilp$-adic shuffle relation (\cite[Theorem 6.4]{seki17}, \cite[Proposition 3.4.3 (ii)]{jarossay19}), $\ilp$-adic duality (\cite[Theorem 1.3]{seki19}) and $\ilp$-adic cyclic sum formula (\cite[Theorem 5.7]{kawasaki19}). These relations are equivalent to the special cases of Theorems \ref{stadic_harmonic_relation}, \ref{stadic_shuffle_relation}, \ref{stadic_duality} and \ref{stadic_csf}, respectively, under Conjecture \ref{refined_kz_conj}.
\end{Rem}


\begin{thebibliography}{HMO}
\bibitem[D]{drinfeld91} V.~G.~Drinfel'd, \emph{On quasitriangular quasi-Hopf algebras and a group closely connected with $\mathrm{Gal}(\overline{\bbQ}/\bbQ)$}, Leningrad Math.~J.~\textbf{2} (1991), 829--860.
\bibitem[Hi]{hirose20} M.~Hirose, \emph{Double shuffle relations for refined symmetric multiple zeta values}, Doc.~Math.~\textbf{25} (2020), 365--380.
\bibitem[HMO]{hmo21} M.~Hirose, H.~Murahara and M.~Ono, \emph{On variants of symmetric multiple zeta-star values and the cyclic sum formula}, Ramanujan J.~(2021).
\bibitem[Ho1]{hoffman97} M.~E.~Hoffman, \emph{The algebra of multiple harmonic series}, J.~Algebra \textbf{194} (1997), 477--495.
\bibitem[Ho2]{hoffman99} M.~E.~Hoffman, \emph{Quasi-shuffle products}, J.~Algebraic Combin.~\textbf{11} (2000), 49--68.
\bibitem[IKZ]{ikz06} K.~Ihara, M.~Kaneko and D.~Zagier, \emph{Derivation and double shuffle relations for multiple zeta values}, Compos.~Math.~\textbf{142} (2006), 307--338.
\bibitem[J1]{jarossay17} D.~Jarossay, \emph{An explicit theory of $\pi_{1}^{\mathrm{un},\mathrm{crys}}(\mathbb{P}^1-\{0,\mu_{N},\infty\})$ - II-1 : Standard algebraic equations of prime weighted multiple harmonic sums and adjoint multiple zeta values}, preprint, \href{https://arxiv.org/abs/1412.5099v3}{\texttt{arXiv:1412.5099v3}}.
\bibitem[J2]{jarossay19} D.~Jarossay, \emph{Adjoint cyclotomic multiple zeta values and cyclotomic multiple harmonic values}, preprint, \href{https://arxiv.org/abs/1412.5099v5}{\texttt{arXiv:1412.5099v5}}.
\bibitem[KXY]{kxy20} M.~Kaneko, C.~Xu and S.~Yamamoto, \emph{A generalized regularization theorem and Kawashima's relation for multiple zeta values}, J.~Algebra \textbf{580} (2021), 247--263.
\bibitem[KZ]{kz21} M.~Kaneko and D.~Zagier, \emph{Finite multiple zeta values}, in preparation.
\bibitem[Ka]{kawasaki19} N.~Kawasaki, \emph{Hyperlogarithms, Bernoulli polynomials, and related multiple zeta values}, Doctoral dissertation in Tohoku University, 2019.
\bibitem[Ko]{komori21} Y.~Komori, \emph{Finite multiple zeta values, symmetric multiple zeta values and unified multiple zeta functions}, Tohoku Math.~J.~\textbf{73} (2021), 221--255.
\bibitem[MO]{mo21} H.~Murahara and M.~Ono. \emph{Yamamoto's interpolation of finite multiple zeta and zeta-star values}, Tokyo J.~Math.~Advance Publication (2021), 1--28.
\bibitem[OSY]{osy21} M.~Ono, S.~Seki and S.~Yamamoto, \emph{Truncated $t$-adic symmetric multiple zeta values and double shuffle relations}, Res.~Number Theory \textbf{7} (2021), 15.
\bibitem[OW]{ow06} Y.~Ohno and N.~Wakabayashi, \emph{Cyclic sum of multiple zeta values}, Acta Arithmetica \textbf{123} (2006), 289--295.
\bibitem[Re]{reutenauer93} C.~Reutenauer, \emph{Free Lie Algebras}, Oxford Science Publications, Oxford, 1993.
\bibitem[Ro]{rosen19} J.~Rosen, \emph{The completed finite period map and Galois theory of supercongruences}, Int.~Math.~Res.~Not.~IMRN 2019, no.~23, 7379--7405.
\bibitem[S1]{seki17} S.~Seki, \emph{Finite multiple polylogarithms}, Doctoral dissertation in Osaka University, 2017.
\bibitem[S2]{seki19} S.~Seki, \emph{The $\ilp$-adic duality for the finite star-multiple polylogarithms}, Tohoku Math.~J.~\textbf{71} (2019), 111--122.
\bibitem[TT]{tt23} K.~Tasaka and Y.~Takeyama, \emph{Supercongruences of multiple harmonic $q$-sums and generalized finite/symmetric multiple zeta values}, Kyushu J.~Math.~\textbf{77} (2023), 75--120.
\bibitem[Y]{yamamoto13} S.~Yamamoto, \emph{Interpolation of multiple zeta and zeta-star values}, J.~Algebra \textbf{385} (2013), 102--114.
\end{thebibliography}
\end{document}